\documentclass[11pt]{amsart}

\usepackage{amsmath,amsthm,amsfonts,amssymb,amscd,euscript,enumerate}
\usepackage{tikz}
\usepackage{tikz-cd}
\usepackage[all,arc]{xy}
\usepackage{enumerate}
\usepackage{mathrsfs} 
\usepackage{thmtools}
\usepackage{datetime}
\usepackage[colorlinks, citecolor = darkspringgreen, linkcolor = melon]{hyperref}
\usepackage{graphicx}
\usepackage{color}
\usepackage{cleveref}
\Crefname{enumi}{}{}
\usepackage{comment}
\usepackage{float}
\input{xy}
\xyoption{all}
\usetikzlibrary{shapes}
\usetikzlibrary{tikzmark,calc}
\usepackage{mathtools}
\usepackage{listofitems}
\usepackage{wrapfig}
\usepackage{caption}
\usepackage{sidecap}
\usepackage{stmaryrd}
\usepackage{url}
\usetikzlibrary{decorations.pathreplacing, arrows.meta, calc}
\usepackage{enumitem}
\usetikzlibrary{patterns.meta}

\definecolor{melon}{rgb}{0.4, 0.2, 1}
\definecolor{maroon}{rgb}{0.5,0,0}
\definecolor{violet}{rgb}{0.6, 0.3, 0.7}
\definecolor{darkviolet}{rgb}{0.35, 0, 0.6}
\definecolor{darkspringgreen}{rgb}{0.14, 0.6, 0.2}
\definecolor{amethyst}{rgb}{0.6, 0.4, 0.8}

\calclayout

\setboolean{@twoside}{false}

\newtheorem{thm}{Theorem}[section]
\newtheorem{cor}[thm]{Corollary}
\newtheorem{prop}[thm]{Proposition}
\newtheorem{lem}[thm]{Lemma}
\newtheorem{conj}[thm]{Conjecture}

\newtheorem{prob}[thm]{Problem}

\theoremstyle{definition}
\newtheorem{defn}[thm]{Definition}

\newtheorem{exmp}[thm]{Example}

\newtheorem*{conjs*}{Conjectures}

\theoremstyle{remark}
\newtheorem{rem}[thm]{Remark}

\makeatletter
\let\c@equation\c@thm
\makeatother
\numberwithin{equation}{section}

\makeatletter
\newcommand*\bigcdot{\mathpalette\bigcdot@{.5}}
\newcommand*\bigcdot@[2]{\mathbin{\vcenter{\hbox{\scalebox{#2}{$\m@th#1\bullet$}}}}}
\makeatother

\makeatletter
\def\subsection{\@startsection{subsection}{3}%
  \z@{.5\linespacing\@plus.7\linespacing}{.1\linespacing}%
  {\bfseries}}
\makeatother

\newcommand{\Z}{\mathbb{Z}}

\newcommand{\R}{\mathbb{R}}

\newcommand{\calP}{\mathcal{P}}

\newcommand{\frakd}{\mathfrak{d}}

\newcommand{\frakD}{\mathfrak{D}}
\newcommand{\frakf}{\mathfrak{f}}

\newcommand{\GG}{\sigma}
\renewcommand{\phi}{\varphi}

\newcommand{\ii}{i}
\newcommand{\jj}{j}

\newcommand{\pp}{p}
\newcommand{\qq}{q}
\newcommand{\fph}{u} 
\newcommand{\fpv}{v} 
\newcommand{\fpi}{r} 
\newcommand{\fpj}{s} 
\DeclareMathOperator{\sh}{sh}
\DeclareMathOperator{\rsh}{rsh}

\DeclareMathOperator{\cut} { \setminus}

\DeclareMathOperator{\CG}{CG}
\newcommand{\CGbc}{\mathrm{CG}_{\mathit{b},\mathit{c}}}

\DeclareMathOperator{\rd}{\pi}
\newcommand{\east}{\mathbf{E}}
\newcommand{\north}{\mathbf{N}}
\newcommand{\tvco}{\varrho} 
\newcommand{\pt}{\mathbf{P}} 
\newcommand{\edge}{e} 
\newcommand{\fclu}[2]{\mathfrak{f}_{\mathbb{R}_{\leq 0}(#1, #2)}} 
\newcommand{\ftrop}[2]{\mathfrak{f}^\circ_{\mathbb{R}_{\leq 0}(#1, #2)}} 

\title{Cluster scattering diagrams via quiver moduli and tight gradings}

\author{Amanda Burcroff}
\address{Amanda Burcroff, Department of Mathematics, MIT, Cambridge, MA 02139, USA}
\email{\href{mailto:amandabu@mit.edu}{amandabu@mit.edu}}

\author{Kyungyong Lee}
\address{Kyungyong Lee, Department of Mathematics, University of Alabama, Tuscaloosa, AL 35487, USA and Korea
Institute for Advanced Study, Seoul 02455, Republic of Korea}
\email{\href{mailto:kyungyong.lee@ua.edu}{kyungyong.lee@ua.edu} \& \href{mailto:klee1@kias.re.kr}{klee1@kias.re.kr}}

\author{Lang Mou}
\address{Lang Mou, Department of Mathematics, University of California Davis, One Shields Avenue, Davis, CA 95616, USA}
\email{\href{mailto:lmou.math@gmail.com}{lmou.math@gmail.com}}

\author{Gregg Musiker}
\address{Gregg Musiker, School of Mathematics, University of Minnesota, Minneapolis, MN 55455, USA}
\email{\href{mailto:musiker@umn.edu}{musiker@umn.edu}}

\author{Markus Reineke}
\address{Markus Reineke, Faculty of Mathematics, Ruhr-Universit\"at Bochum,  Universit\"atsstra{\ss}e 150, 44780 Bochum, Germany}
\email{\href{mailto:markus.reineke@ruhr-uni-bochum.de}{markus.reineke@ruhr-uni-bochum.de}}

\date{\today}

\begin{document}

\begin{abstract}
    We study rank-2 cluster scattering diagrams through moduli spaces of quiver representations and a recently developed combinatorial framework of tight gradings. Combining quiver-theoretic and combinatorial methods, we prove and extend a collection of conjectures posed by Elgin--Reading--Stella concerning the structural and enumerative properties of the wall-function coefficients. The tight grading perspective also provides a new proof of the Weyl group symmetry of the scattering diagram.
\end{abstract}


\maketitle

\setcounter{tocdepth}{1}
\tableofcontents

\vspace{-1em}

\section{Introduction}

\subsection{Background and motivation}

Scattering diagrams \cite{KS, GS} are combinatorial structures encoding wall-crossing phenomena arising in various sources of mathematics. Cluster scattering diagrams of Gross--Hacking--Keel--Kontsevich \cite{GHKK} are a special class attached to cluster algebras/varieties \cite{FZ1,FG1}, where the wall-functions reflect their mutation structure and positivity properties. In rank 2, they are closely related to the \emph{tropical vertex} of Gross--Pandharipande--Siebert, where the wall-functions are expressed through generating series of relative Gromov-Witten invariants of toric surfaces \cite{GPS} and Euler characteristics of quiver moduli \cite{RCQM,RW}.

From the cluster algebra perspective, the coefficients of rank-2 wall-functions lead to
the Laurent positivity in rank-2 cluster algebras $\mathcal{A}(b,c)$ of Fomin--Zelevinsky \cite{Zel}.  This positivity admits both combinatorial descriptions via compatible gradings using edges in Dyck paths \cite{LS,LLZ} and more geometric ones via scattering diagrams \cite{GHKK, KY}.

Based on extensive computational evidence, Elgin, Reading, and Stella \cite{ERS} formulated a collection of conjectures regarding the wall-function coefficients $\tau(i,j)$ of the cluster scattering diagram $\frakD_{(b,c)}$ associated with $\mathcal{A}(b,c)$. These conjectures, which primarily concern the dependence of $\tau(i,j)$ on the positive integers $b$ and $c$, form the starting point of the present work. Tight gradings, recently introduced in \cite{BLM,BLM2}, refine the framework of compatible gradings \cite{LLZ,Rupgengreed}, and provide a unified combinatorial model for these coefficients.
Our approach uses both quiver representation techniques and the combinatorics of tight gradings to confirm many of these predictions and to derive further generalizations.

\subsection{Elgin--Reading--Stella's conjectures}

For coprime positive integers $d$ and $e$, write the wall-function on the ray $\mathbb{R}_{\leq 0}(db, ec)$ in the cluster scattering diagram $\frakD_{(b,c)}$ as the power series
\[
    \fclu{db}{ec} = 1 + \sum \tau(i, j)x_1^{ib}x_2^{jc},
\]
where the sum is over positive integers $i$ and $j$ such that $(i,j) = (kd, ke)$ 
for some $k \geq 1$. Note that in the case that $(db,ec)$ is not a primitive vector, we can re-express $\mathbb{R}_{\leq 0} (db,ec) = \mathbb{R}_{\leq 0} (v_1,v_2)$ for $\gcd(v_1,v_2)=1$. See \eqref{eq: def of tau(i,j)} for how these $\tau(i,j)$ are precisely defined.

Elgin, Reading, and Stella posed the following conjectures and problem.

\begin{conjs*}[{\cite[Conjecture 1--18]{ERS}}] Let $g = \gcd(ib, jc)/\gcd(i, j)$.
    \leavevmode
    \begin{enumerate}[label=\textbf{C\arabic*}., ref= \textbf{C\arabic*}, topsep=0ex, itemsep=0.5ex]
    \item \label{C1} For $\ii,\jj>0$, the coefficient $\tau(\ii,\jj)$ is a polynomial in $b$, $c$, and $g$.
    \item \label{C2} The polynomial has $g$ as a factor and its degree in $g$ is $\gcd(\ii,\jj)$.
    \item \label{C3} Its degree in $b$ is $\jj-1$ and its degree in $c$ is $\ii-1$.
    \item \label{C4} The polynomial $(\max(\ii,\jj))! \tau(\ii ,\jj)$ has integer coefficients.
    \item \label{C5} $\displaystyle \tau(1,\jj)=\frac{g}{b}\binom{b}{\jj}$
    \item \label{C6} $\displaystyle \tau(\ii ,1)=\frac{g}{c}\binom{c}{\ii}$
    \item \label{C7} $\displaystyle \tau(\ii ,\ii ) = \dfrac{g}{(b-1)(c-1)\ii  + g}\binom{(b-1)(c-1)\ii  + g}{\ii }$
    \item \label{C8} $\displaystyle \tau(\ii ,\ii ) = \sum_{\ell=0}^{\infty} \frac{g}{\ell+1}\binom{\ii -1}{\ell}\binom{\ii (bc-b-c) + g-1}{\ell}$
    \item \label{C9} $\displaystyle \tau(\ii ,\ii -1)=\frac{g}{\ii (\ii b-\ii +1)}\binom{(\ii b-\ii +1)(c-1)}{\ii -1}$
    \item \label{C10} $\displaystyle \tau(\jj-1,\jj)=\frac{g}{\jj(\jj c-\jj+1)}\binom{(\jj c-\jj+1)(b-1)}{\jj-1}$
    \item \label{C11} $\tau^{b,b}(\ii ,\jj)$ is a polynomial in $b$ of degree $\ii +\jj-1$ that expands positively in the basis $\left\{\binom{b}{0}, \binom{b}{1}, \binom{b}{2}, ...\right\}$.
    \item \label{C12} $\tau^{b,b}(\ii,\jj)$ has unimodal log-concave coefficients.
    \item \label{C13} $\displaystyle \tau^{1,5}(2\jj,\jj) = \frac{1}{\jj}\sum_{\ell=0}^{\infty} \binom{\ell}{\jj-\ell+1}\binom{\jj+\ell-1}{\ell}$
    \item \label{C14} $\tau(\ii,\jj;k)$ is a polynomial of degree $\jj-k$ in $b$ and degree $\ii-k$ in $c$ and has a term that is a nonzero constant times $b^{\jj-k} c^{\ii-k}$.
    \item \label{C15} If $\gcd(\ii,\jj)=1$ then $\displaystyle \tau(\ii k,\jj k;k)=\frac{\tau(\ii,\jj;1)^k}{k!}$.
    \item \label{C16} $\displaystyle \tau(k,\jj k;k-1)=\frac{\tau(1,\jj;1)^{k-1}\cdot p_\jj}{(k-2)!}$, where $p_\jj$ is a polynomial in $b$ and $c$ that depends only on $\jj$, not $k$.
    \item \label{C17} $\displaystyle \tau(\ii k,k;k-1)=\frac{\tau(\ii,1;1)^{k-1}\cdot p_i}{(k-2)!}$, where $p_\ii$ is a polynomial in $b$ and $c$ that depends only on $\ii$, not $k$.
    \item \label{C18} If $\gcd(\ii,\jj)=1$, then $\tau(\ii k,\jj k;k-1)$ has a factor $p_{\ii,\jj}$  that depends only on $\ii$ and $\jj$, not on $k$, and the other factors of $\tau(\ii k,\jj k;k-1)$ also appear as factors of $\tau(\ii,\jj;1)$.
\end{enumerate}
\end{conjs*}

\begin{prob}[{\cite[Section 4]{ERS}}]\label{problem_of_ERS}
    Understand how the Weyl group symmetry appears in the combinatorics of tight gradings.
\end{prob}

\subsection{Main results}

The above conjectures regard the cluster scattering diagram $\frakD_{(b,c)} = \frakD(1+x_1^b, 1+x_2^c)$, which is closely tied with another scattering diagram $\frakD^\circ_{(c,b)} = \frakD((1+x_1)^c, (1+x_2)^b)$ via the change-of-lattice transformation; see \Cref{section: change of lattice}. (Beware the swap between $b$ and $c$.) 

Let $d$ and $e$ be positive integers with $\gcd(d,e)=1$.  The wall-function of $\frakD^\circ_{(c,b)}$ on the wall $\mathbb{R}_{\leq 0}(d,e)$ is written as
\[
    \ftrop{d}{e} = 1 + \sum_{k\geq 1} \tvco(kd, ke) x_1^{kd} x_2^{ke}.
\]
With $d$ and $e$ fixed, write $\tau_k = \tau(kd, ke)$ and $\tvco_k = \tvco(kd, ke)$. We summarize in the theorem below the structural properties obtained for $\tau_k$ and $\tvco_k$. The proof will be provided in \Cref{section: proof}.

\newpage 
\begin{thm}\label{thm: main 1 intro}
    \leavevmode
    \begin{enumerate}[font = \upshape, topsep=0ex]
        \item There are nonnegative integers $\lambda(\pt_1, \pt_2)$ associated to pairs of integer 
        partitions, here $\pt_i = (p_1,p_2,\dots, p_\ell)$ with $p_1 \geq p_2 \geq \ldots p_\ell \geq 1$,
        such that
        \begin{equation}\label{eq: tvco main}
            \tvco_k = \sum_{\pt_1\vdash kd}\sum_{\pt_2\vdash ke} \lambda(\pt_1,\pt_2) \prod_{p\in \pt_1} \binom{c}{p}\prod_{q\in \pt_2} \binom{b}{q} \in bc\cdot \mathbb{Q}[b,c],
        \end{equation}
        and $\tvco_k$ has a leading term $c^{kd}b^{ke}$ with positive coefficient. 
        \item Let $g=\gcd(db,ec)$. Then $\tau_k$'s are expressed in terms of $\tvco_k$'s by
        \begin{align*}
            \tau_k & =  \sum_{\ell=1}^k \binom{g/(bc)}{\ell} \sum_{k_1+\cdots+k_\ell = k} \tvco_{k_1}\cdots \tvco_{k_\ell} \\
            & = g\sum_{\ell=1}^k \frac{(g-bc)\cdots(g-(\ell-1)bc)}{\ell!}\sum_{k_1+\cdots+k_\ell = k} \frac{\tvco_{k_1}}{bc}\cdots \frac{\tvco_{k_\ell}}{bc}.
        \end{align*}
        Therefore every $\tau_k$ is in $\mathbb{Q}[b,c,g]$. Its degree in $g$ is $k$ by the positivity in \emph{\Cref{thm: main 1 intro}(1)}.
        \item Write $\tau_k$ as a polynomial in $g$:
        \[
            \tau_k = \sum_{n=1}^k \tau_k(n) g^n, \quad \tau_k(n)\in \mathbb{Q}[b,c].
        \]
        Then each $\tau_k(n)$ has a leading term $b^{ke-n}c^{kd-n}$ with positive coefficient. 
        \item For any $n=1, \dots, k$, we have
        \[
            \tau_k(n) = \frac{1}{n!}\sum_{k_1+\cdots+k_n=k} \tau_{k_1}(1)\cdots \tau_{k_n}(1).
        \]
    \end{enumerate}
\end{thm}

In particular, we have the following consequences:
\begin{itemize}
    \item \Cref{C1,C2} are implied by part (2) of \Cref{thm: main 1 intro}.
    \item \Cref{C3,C14} are implied by part (3) of \Cref{thm: main 1 intro}.
    \item \Cref{C5,C6} are obtained using \Cref{thm: main 1 intro}(2) in the case $k=1$ and computing the only $\lambda(\pt_1, \pt_2)$ appearing in \eqref{eq: tvco main}, where both partitions are of length one (\Cref{cor: length one partitions}).
    \item \Cref{C15,C18} are special cases of \Cref{thm: main 1 intro}(4) when $n=k$ and $n=k-1$, respectively.
    \item \Cref{C16,C17} follow from combining \Cref{thm: main 1 intro}(4) with a case study in \Cref{ex: d=1 e=j}.
\end{itemize}

Both $\lambda(\pt_1, \pt_2)$ and $\tau_k$ admit enumerative interpretations in terms of \emph{tight gradings}; see \Cref{section: tight grading} for a detailed treatment. Tight gradings were introduced in \cite{BLM, BLM2} to give a combinatorial interpretation of rank-$2$ scattering diagram coefficients. Each tight grading is assignment of nonnegative integer weights to the edges of a Dyck path satisfying certain conditions (see \Cref{def: tight grading}).  They form a subset of the compatible gradings defined by Lee--Li--Zelevinsky \cite{LLZ} (and extended by Rupel \cite{Rupgengreed}) to study the \emph{greedy bases} of (generalized) cluster algebras.  The conditions under which a compatible grading is tight are directly motivated by rank-$2$ cluster scattering diagrams.  Thus tight gradings serve as a bridge between the Lee--Schiffler proof of Laurent positivity using Dyck path combinatorics \cite{LSpositive} and the Gross--Hacking--Keel--Kontsevich proof via scattering diagrams \cite{GHKK}. 

The interpretation of $\lambda(\pt_1, \pt_2)$ and $\tau_k$ using tight gradings plays an important role in our analysis in \Cref{section: proof}.  In \Cref{picto_tight}, we give a new description of tight gradings as arrangements of rectangular tiles, called \emph{footprints}.  A similar tiling interpretation was outlined in \cite{BLM2}, but the footprint framework admits more tiles and we believe better captures the relative ``tightness'' of tight gradings.

Quiver moduli are algebraic varieties parametrizing isomorphism classes of stable representations of quivers. Despite being interesting in their own right for encoding hard representation-theoretic classification problems, it turned out that their topological invariants encode enumerative invariants such as Gromov-Witten and Donaldson-Thomas invariants. A typical example is the refined GW/Kronecker correspondence \cite{RW}: both types of geometric invariants are encoded in the so-called tropical vertex, now called a rank $2$ (generalized) cluster scattering diagram (see the following Section \ref{section: change of lattice}). This allows us to translate the above conjectures into statements on Euler characteristics of quiver moduli spaces, and to readily apply results obtained in \cite{RW} (in fact, we will therefore never need the precise definition of (stable) quiver representations and their moduli spaces). This method is particularly helpful in obtaining closed formulas for $\tau(i, j)$, as in \Cref{C5} through \Cref{C10}. In particular, Conjectures \Cref{C7,C8} above are equivalent to a conjecture of Gross and Pandharipande \cite[Conjecture (1,4)]{GP}, proved in \cite[Theorem 9.4]{RW}.

The next theorem lists the subset of conjectures of Elgin--Reading--Stella that we are able to prove combining the aforementioned techniques.

\begin{thm}\label{thm: main 2 intro}
    \emph{\Cref{C1}, \Cref{C2}, \Cref{C3}, \Cref{C5}-\Cref{C10}, \Cref{C14}-\Cref{C18}} are true.
\end{thm}

\begin{rem}
    Ryota Akagi \cite{Akagi2} has recently given proofs of \Cref{C1}, \Cref{C2}, \Cref{C3}, \Cref{C5}, \Cref{C6}, \Cref{C15} and obtained partial results on \Cref{C14}, \Cref{C16}, \Cref{C17}, \Cref{C18}. His method is different from ours and builds on his earlier work \cite{Akagi} on dilogarithm identities. He further proves \Cref{C11} \cite[Proposition 4.3]{Akagi2}, which our method does not address.
\end{rem}

Note that Conjectures \Cref{C4}, \Cref{C12}, and \Cref{C13} remain unsolved. In fact, the conjectural formula \Cref{C13} is mysterious from both the quiver moduli and tight grading points of view.

Finally, we give an answer to \Cref{problem_of_ERS}, using the Lee--Li--Zelevinsky mutation operation on gradings \cite{LLZ} and a retraction operation, both discussed in \Cref{section: weyl symmetry}.

\begin{thm}\label{thm: weyl symmetry}
    The composition of mutation and retraction on tight gradings is a bijection exhibiting the Weyl group symmetry, namely that
    $$\tau(\ii,\jj) = \tau(\ii,b\ii-\jj) \text{ and } \tau(\ii,\jj) = \tau(c\jj - \ii,\jj)\,.$$
\end{thm}

A consequence of \Cref{thm: weyl symmetry} is a new combinatorial proof of the invariant structure of $\frakD_{(b,c)}$ described in \cite[Example 1.15]{GHKK}. Note that the proof of the interpretation of $\tau(i,j)$ in terms of tight gradings in \cite[Section 5.5]{BLM2} does not rely on this invariant structure.

\addtocontents{toc}{\protect\setcounter{tocdepth}{-1}}
\section*{Acknowledgments}
\addtocontents{toc}{\protect\setcounter{tocdepth}{1}}
AB was partially supported by NSF grant DMS-2503411. KL was supported by the University of Alabama, Korea Institute for Advanced Study, the NSF grant DMS 2302620, and the Simons Foundation grant SFI-MPS-SFM-00010895. G.M. was partially funded by Simons Foundation Travel Support for Mathematicians.  The authors are grateful to Ryota Akagi and Salvatore Stella for their correspondences and helpful discussions.

\section{Rank 2 scattering diagrams and the change-of-lattice}\label{section: change of lattice}

We briefly recall some basics of  rank-2 (generalized) cluster scattering diagrams.

Fix a base field $\Bbbk$ of characteristic zero. For $i=1,2$, take power series 
\[
    P_i(x_i) = 1 + \sum_{k\geq 1} p_{i,k}x_i^k \in \Bbbk\llbracket x_i\rrbracket  \quad \text{where} \quad p_{i,k}\in \Bbbk.
\]
By a result of Kontsevich--Soibelman \cite{KS}, the two power series $P_1$ and $P_2$ uniquely determine a \emph{consistent scattering diagram}
\begin{multline}\label{eq: universal sd}
    \frakD(P_1, P_2) = \{(\mathbb{R}(1,0), P_1),(\mathbb{R}(0,1), P_2)\}  \\ 
    \cup \{(\mathbb{R}_{\leq 0}(v_1,v_2), \fclu{v_1}{v_2})\mid (v_1,v_2)\in \mathbb{Z}_{>0}^2,\gcd(v_1,v_2)=1\}.
\end{multline}
Every element in $\frakD(P_1,P_2)$, referred to as a \emph{wall}, is a pair $(\frakd, \frakf)$ where $\frakd$ is either a line or a ray in $\mathbb{R}^2$ and $\frakf$ (the \emph{wall-function}), when $\frakf=\fclu{v_1}{v_2}$, is a power series in $\Bbbk\llbracket x_1^{v_1}x_2^{v_2}\rrbracket $ with constant term $1$. We refer the reader to \cite[Section 1.3]{GPS} for the algebraic details of the \emph{consistency} of scattering diagrams and to \cite[Section 4.1]{BLM} for a presentation closer to the current setup.

The primary focus of this paper is the case where
\[
    P_1(x_1) = 1 + x_1^b \quad \text{and} \quad P_2(x_2) = 1 + x_2^c
\]
for positive integers $b$ and $c$. We denote this scattering diagram by $\frakD_{(b,c)}$. They are known as the rank-2 special class of the \emph{cluster scattering diagrams} of Gross--Hacking--Keel--Kontsevich \cite{GHKK}, and are closely tied with Fomin--Zelevinsky's cluster algebras $\mathcal A(b,c)$ in rank 2 \cite{Zel}.

For $\frakD_{(b,c)}$, it is also known that each $\fclu{v_1}{v_2}$ lives in $\Bbbk\llbracket x_1^{db}x_2^{ec}\rrbracket $ for the unique pair
\[
    (d,e)\in \mathbb{Z}_{>0}^2 \quad \text{such that}\quad \gcd(d,e)=1 \quad \text{and} \quad v_1ec=v_2db.
\]
Therefore we can write
\begin{equation}\label{eq: def of tau(i,j)}
    \fclu{v_1}{v_2} = 1 + \sum_{k\geq 1} \tau^{(b,c)}(kd, ke) x_1^{kdb}x_2^{kec},
\end{equation}
defining the coefficients $\tau(kd, ke)=\tau^{(b,c)}(kd,ke)$, viewed as functions of $b$ and $c$.

\begin{exmp}
Let $b = 2$ and $c = 3$. 

If $(v_1, v_2) = (1, 1),$ 
then $(d,e) = (3,2)$ and 
\[
    \mathfrak{f}_{\mathbb{R}_{\leq 0}(1, 1)} = 1 + \tau(3, 2)x_1^{6}x_2^{6} + \tau(6, 4)x_1^{12}x_2^{12} + \cdots.
\]

On the other hand, if $(v_1, v_2) = (2, 3),$ then $(d, e) = (1, 1)$, and
\[
    \mathfrak{f}_{\mathbb{R}_{\leq 0}(2, 3)} = 1 + \tau(1, 1)x_1^{2}x_2^{3} + \tau(2, 2)x_1^{4}x_2^{6} + \cdots.
\]

As another example, if $(v_1, v_2) = (3, 5),$
then $(d, e) = (9, 10)$, and
\[
    \mathfrak{f}_{\mathbb{R}_{\leq 0}(3, 5)} = 1 + \tau(9,10)x_1^{18}x_2^{30} + \tau(18, 20)x_1^{36}x_2^{60} + \cdots.
\]
In all three of these cases, note that both $(v_1, v_2)$ and $(d, e)$ are primitive such that $\frac{v_1}{v_2} = \frac{db}{ec}$.
\end{exmp}

A closely related scattering diagram is
\[
    \frakD^\circ_{(c,b)} \coloneqq \frakD((1+x_1)^c, (1+x_2)^b).
\]
(Notice the swap between $b$ and $c$.) For positive $d$ and $e$ with $\gcd(d,e)=1$, denote the wall-function on $\mathbb{R}_{\leq 0}(d,e)$ by
\begin{equation}\label{eq: wall-func gps}
    \ftrop{d}{e} = 1 + \sum_{k\geq 1}\tvco^{(c,b)}(kd,ke)x_1^{kd}x_2^{ke}\in \Bbbk\llbracket x_1^{d}x_2^{e}\rrbracket ,
\end{equation}
defining the coefficients $\tvco(kd, ke) = \tvco^{(c,b)}(kd,ke)$, viewed as functions of $c$ and $b$.

As pointed out in \cite[Example 1.15]{GHKK} (and stated in \Cref{prop: change of lattice}), the scattering diagram $\frakD^\circ_{(c,b)}$ can be transformed into $\frakD_{(b,c)}$ by the \emph{change-of-lattice trick}, given in \cite[Proof of Proposition C.13, Step IV]{GHKK}. A more detailed proof of \Cref{prop: change of lattice} can be found in \cite[Proposition 1.12]{GL}.

\begin{prop}[Change-of-lattice]\label{prop: change of lattice}
    For each $(d,e)\in \mathbb{Z}_{>0}^2$, define 
    \[
        g=g(d,e,b,c)\coloneqq \gcd(db,ec)/\gcd(d,e).
    \]
    Then we have for $\gcd(d, e) = 1$,
    \[
        \ftrop{d}{e} = \left(1 + \sum_{k\geq 1}\tau(kd,ke)x_1^{kd}x_2^{ke}\right)^{bc/g}.
    \]
\end{prop}

This is an important tool that we use throughout the paper, especially in \Cref{section: wall function quiver moduli,section: proof}.

\section{Wall-functions via quiver moduli}\label{section: wall function quiver moduli}

In this section, we prove \Cref{C1}, \Cref{C5}, \Cref{C6}, \Cref{C7}, \Cref{C8}, \Cref{C9}, and \Cref{C10} by the quiver method. 

Namely, Proposition~\ref{prop: change of lattice} allows us to translate between the notation and conventions of \cite{ERS} and of \cite{RW}. 
We define $S_{(d,e)}(t)$ by
$$S_{(d,e)}(t)=\sum_{k\geq 0}\tau(kd,ke)t^k\in\mathbb{Q}\llbracket t\rrbracket $$
and readily find
\begin{cor}\label{cor20250702n1} We have
    $$S_{(d,e)}(t)=f_{(d,e)}(t)^{g/bc}.$$
\end{cor}
\begin{proof}
Set $l_1=b$, $l_2=c$ , and consider the series $f_{(d,e)}(t)\in\mathbb{Q}\llbracket t\rrbracket $ appearing in the tropical vertex factorization 
\cite[Section 2, (1)]{RW}, where we specialize all $s_k$ and all $t_l$ to $t$, and specialize $x$ and $y$ to $1$. Then $f_{(d,e)}(t) = \ftrop{d}{e}\vert_{x_1^dx_2^e = t}$ while $S_{(d,e)}(t) = \fclu{db}{ec}\vert_{x_1^{db}x_2^{ec} = t}$. Proposition~\ref{prop: change of lattice}  implies the desired identity.
\end{proof}

For the details of the following notions, we refer to \cite{RW} (in fact, as mentioned earlier, we do not need these notions in detail here). We consider the complete bipartite quiver $K(b,c)$ with arrows from a set of $c$ vertices to a set of $b$ vertices. Given a dimension vector $(\mathbb{P}_1,\mathbb{P}_2)$ with $\mathbb{P}_1=(p_{1,1},\ldots,p_{1,b})$, $\mathbb{P}_2=(p_{2,1},\ldots,p_{2,c})$ for $K(b,c)$, let $M^{\rm st}(\mathbb{P}_1,\mathbb{P}_2)$ be the moduli space parametrizing isomorphism classes of stable representations of $K(b,c)$ of dimension vector $(\mathbb{P}_1,\mathbb{P}_2)$ with respect to the non-trivial symmetric stability, and define the following sum of topological Euler characteristics:
$$\chi_{(d,e)}(k)=\sum\chi(M^{\rm st}(\mathbb{P}_1,\mathbb{P}_2)),$$
the sum ranging over all $\mathbb{P}_1$ summing up to $kd$, and all $\mathbb{P}_2$ summing up to $ke$.

Then \cite[Theorem 8.1]{RW} immediately implies  that the series $f_{(d,e)}(t)$ is determined by a functional equation:

\begin{thm}\label{thm20250702n1} The series $f_{(d,e)}(t)\in\mathbb{Q}\llbracket t \rrbracket$ is determined by
   $$f_{(d,e)}(t)=\prod_{k\geq 1}(1-(tf_{(d,e)}(t)^E)^k)^{-k\chi_{(d,e)}(k)},$$ 
where $E=(bcde-bd^2-ce^2)/bc$.
\end{thm}

Combining the previous results, we note:

\begin{cor}\label{lem20250702n2}
    $S_{(d,e)}(t)$ is determined by the functional equation
$$S_{(d,e)}(t)=\prod_{k\geq 1}(1-(tS_{(d,e)}(t)^{bcE/g})^k)^{-k\chi_{(d,e)}(k)g/bc}.$$
\end{cor}
\begin{proof}
 This immediately follows from   Corollary~\ref{cor20250702n1} and Theorem~\ref{thm20250702n1}.
\end{proof}

We will use this characterization below to prove \Cref{C1}. But first we prove \Cref{C5}--\Cref{C10}.

We consider the $t^1$-coefficients of $S_{(d,e)}(t)$ and $f_{(d,e)}(t)$ to find, using \cite[Corollary 9.1]{RW}:

If $\ii$ and $\jj$ are coprime, we have
$$\tau(\ii,\jj)=\frac{g}{bc}\chi_{(\ii,\jj)}(1).$$

In particular, \cite[Theorem 9.4]{RW} already solves \Cref{C9,C10}. 
\begin{prop}[{\cite[Theorem 9.4]{RW}}]
We have
    \begin{align*}
    & \tau(\ii,\ii-1)=\frac{g}{\ii((b-1)\ii+1)}\binom{(b-1)(c-1)\ii+c-1}{\ii-1},\\ 
    & \tau(\jj-1,\jj)=\frac{g}{\jj(\jj c-\jj+1)}\binom{(\jj c-\jj+1)(b-1)}{\jj-1}.
    \end{align*}
\end{prop}

Regarding \Cref{C5,C6}, we give two proofs. The first proof is given below, and the second proof follows from the more general formula in \Cref{prop: formula one length one part}. 

\begin{prop}\label{prop: central wall}
We have
    $$\tau(1,\jj)=\frac{g}{b}\binom{b}{\jj}\quad \text{and} \quad \tau(\ii,1)=\frac{g}{c}\binom{c}{\ii}.$$
\end{prop}\begin{proof}
    If $\ii$=1, the tuple $\mathbb{P}_1$ is given by a single entry $1$ at one of the $c$ vertices. The moduli space $M^{\rm st}(\mathbb{P}_1,\mathbb{P}_2)$ is non-empty only if $\mathbb{P}_2$ consists only of entries $0$ or $1$, in which case it reduces to a point. We thus have $c\binom{b}{\jj}$ choices for such a dimension vector contributing to the sum defining $\chi_{(1,\jj)}(1)$, and thus
$$\chi_{(1,\jj)}(1)=\frac{g}{bc}c\binom{b}{\jj}=\frac{g}{b}\binom{b}{\jj},$$
confirming \Cref{C5}. Again, \Cref{C6} follows by duality.
\end{proof}

The following proposition proves \Cref{C7,C8}.
\begin{prop}
We have
    $$\tau(\ii,\ii)=\frac{g}{(b-1)(c-1)\ii+g}\binom{(b-1)(c-1)\ii+g}{\ii} = \sum_{\ell=0}^{\infty} \frac{g}{\ell+1}\binom{\ii-1}{\ell}\binom{\ii(bc-b-c)+g-1 }{\ell}.$$
\end{prop}\begin{proof}
    We consider the central slope $\ii=\jj$. We repeat the methods of \cite[Section 11]{RW}. First we find $\chi_{(1,1)}(k)=0$ for all $k\geq 2$, and $\chi_{(1,1)}(1)=bc$.  The defining functional equation for $S_{(1,1)}(t)$ thus simplifies to
$$S_{(1,1)}(t)=(1-tS_{(1,1)}(t)^{(bc-b^2-c^2)/g})^{-g}.$$
Specializing \cite[Theorem 1.4]{RNCH} to $q=1$, we find the following: If a formal series $F(t)\in\mathbb{Q}\llbracket t \rrbracket $ with constant term one is determined by 
$$F(t)=(1-tF^{(m-1)/n})^{-n},$$
then it is given by
$$F(t)=\sum_{i\geq 0}\frac{n}{(m-1)i-n}\binom{mi+n-1}{i}t^i.$$
This applies to $F(t)=S_{(1,1)}(t)$ using $m=(b-1)(c-1)$ and $n=g$, yielding
$$\tau(\ii,\ii)=\frac{g}{(bc-b-c)\ii+g}\binom{(b-1)(c-1)\ii+g-1}{\ii}=\frac{g}{(b-1)(c-1)\ii+g}\binom{(b-1)(c-1)\ii+g}{\ii},$$
confirming \Cref{C7} \cite[Conjecture 7]{ERS}. As already noted there, \Cref{C8} follows.
\end{proof}

Now we return to the general case. We first analyze the dependence of $\chi_{(d,e)}(k)$ on $b$ and $c$. 
\begin{prop}
There exists a polynomial $P_{(d,e),k}(b,c)$ of degree at most $kd-1$ in $b$ and $ke-1$ in $c$ such that
    $$\chi_{(d,e)}(k)=bcP_{(d,e),k}(b,c).$$
\end{prop}
\begin{proof}

The quiver $K(b,c)$ has a natural $S_b\times S_c$-symmetry, which implies that the moduli space $M^{\rm st}(\mathbb{P}_1,\mathbb{P}_2)$ only depends on the entries of $\mathbb{P}_1$, $\mathbb{P}_2$, but not on their respective ordering. This motivates the following definition: for partitions $\lambda_1$, $\lambda_2$, we define $M^{\rm st}(\lambda_1,\lambda_2)$ as the moduli space of stable representations for the dimension vector $(\lambda_1,\lambda_2)$ of $K(\ell(\lambda_1),\ell(\lambda_2))$. Here we denote by $\ell(\lambda)$ the length of a partition; we also need the notation $z_\lambda$ for the number of derangements of $\lambda$. Then we find
$$\chi_{(d,e)}(k)=\sum_{\substack{{\lambda_1\vdash kd}\\{\lambda_2\vdash ke}}}\chi(M^{\rm st}(\lambda_1,\lambda_2))z_{\lambda_1}z_{\lambda_2}\binom{b}{\ell(\lambda_1)}\binom{c}{ \ell(\lambda_2)}.$$
Namely, the possible ordered partitions $\mathbb{P}_1$ of $kd$ arise by choosing an unordered partition $\lambda_1$ of $kd$, a derangement of $\lambda_1$, and a choice of $\ell(\lambda_1)$ among the $b$ vertices in $K(b,c)$ (and similarly for $\mathbb{P}_2$). Since $\chi(M^{\rm st}(\lambda_1,\lambda_2))$ is independent of $b$ and $c$, we see that $\chi_{(d,e)}(k)$ behaves polynomially in $b$ and $c$. This polynomial admits a factor $bc$ since $\ell(\lambda_1),\ell(\lambda_2)\geq 1$ for all partitions. The degree is at most $kd$ in $b$ and at most $ke$ in $c$. We can thus write
$$\chi_{(d,e)}(k)=bcP_{(d,e),k}(b,c)$$
for some polynomial $P_{(d,e),k}(b,c)$ of degree at most $kd-1$ in $b$ and $ke-1$ in $c$.
\end{proof}

Now we use techniques of \cite[Section 4]{RCQM}. Essentially using \cite[Lemma 4.3]{RCQM}, we can prove: if a formal series $F(t)\in\mathbb{Q}\llbracket t\rrbracket $ with constant term one is determined by
$$F(t)=\prod_{k\geq 1}(1-(tF(t))^k)^{-a_k},$$
then, for all $k\in\mathbb{Q}$, we have
$$F(t)^k=\sum_{i\geq 0}k\sum_{l_*}\prod_{j\geq 1}\frac{1}{l_j!}a_j\prod_{k=1}^{l_j-1}((k+i)a_j+k)t^i,$$
where the second sum ranges over all tuples $l_*=(l_1,l_2,\ldots)$ of nonnegative integers such that $i=\sum_jjl_j$. Thanks to Lemma~\ref{lem20250702n2}, 
when $F(t)=S_{(d,e)}(t)^{bcE/g}$, $k=g/bcE$ and $a_k=k\chi_{(d,e)}(k)E$, we get
$$S(t)=\sum_{i\geq 0}\frac{g}{bcE}\sum_{l_*}\prod_{j\geq 1}\frac{1}{l_j!}j\chi_{(d,e)}(j)E\prod_{k=1}^{l_j-1}((\frac{g}{bcE}+i)j\chi_{(d,e)}(j)E+l)t^i.$$
Using $\chi_{(d,e)}(k)=bcP_{(d,e),k}(b,c)$, we clear all denominators $bcE$ and find
$$S(t)=\sum_{i\geq 0}g\sum_{l_*}\prod_j\frac{j}{l_j!}P_{(d,e),j}(b,c)\prod_k((i(debc-d^2b-e^2c)+g)P_{(d,e),j}(b,c)+k)t^i.$$
This implies the following closed formula for the $\tau(\ii,\jj)$ in terms of the geometrically defined $P_{(d,e),j}(b,c)$:
$$\tau(\ii,\jj)=g\sum_{l_*}\prod_j\frac{\jj}{l_j!}P_{(d,e),\jj}(b,c)\prod_{k=1}^{l_j-1}((\gcd(\ii,\jj)\cdot(debc-d^2b-e^2c)+g)P_{(d,e),\jj}(b,c)+k),$$
where the sum ranges over all $l_*$ such that $\sum_jjl_j=\gcd(\ii,\jj)$. 
From direct inspection of this expression, we obtain \Cref{C1}:

\begin{prop}[{\cite[Conjecture 1]{ERS}}]
    For $\ii,\jj>0$, the coefficient $\tau(\ii,\jj)$ is  a polynomial in $b,c$, and $g$.
\end{prop}

\section{Tight gradings}\label{section: tight grading}
Tight gradings were introduced in \cite{BLM} to give a combinatorial interpretation of the wall-function coefficients in generalized cluster scattering diagrams. In this section, we recall the definition of tight gradings and establish some of their structural properties. 

\subsection{Maximal Dyck paths}
Fix $d_1,d_2 \in \Z_{\geq 0}$. Consider a rectangle with vertices $(0,0)$, $(0,d_2)$, $(d_1,0)$, and $(d_1,d_2)$ with a main diagonal from $(0,0)$ to $(d_1,d_2)$.

We first require various notions of paths.  

\begin{defn} Let $d_1$ and $d_2$ be nonnegative integers.
\begin{itemize}
\item A \emph{NE path} is a lattice path in $(\mathbb{Z}\times \mathbb{R})\cup(\mathbb{R}\times \mathbb{Z})\subset \mathbb{R}^2$ starting at $(0,0)$ and ending at $(d_1,d_2)$, proceeding by only unit north and east steps.
\item A \emph{Dyck path} is a NE path that never passes strictly above the main diagonal from $(0,0)$ to $(d_1,d_2)$.
\item The \emph{maximal Dyck path} $\calP(d_1,d_2)$ is the Dyck path proceeding from $(0,0)$ to $(d_1,d_2)$ that is closest to the main diagonal.
\end{itemize}
\end{defn}

Given a subset $C$ of steps (also known as \emph{edges}) in a NE path $\calP$, we denote the set of east (resp. north) steps by $C_\east$ (resp. $C_\north$) and the total number of steps by $|C|$.  For edges $\edge_1 \in \calP_\east$ and $\edge_2 \in \calP_\north$, let $\overrightarrow{\edge_1\edge_2}$ denote the subpath proceeding east from the left endpoint of $\edge_1$ to the top endpoint of $\edge_2$, continuing cyclically around $\calP$ if $\edge_1$ is east of $\edge_2$.

The Dyck paths from $(0,0)$ to $(d_1,d_2)$ form a partially ordered set by comparing the heights at all vertices.  The \emph{maximal Dyck path} $\calP(d_1,d_2)$ is the maximal element under this partial order.  We label the horizontal edges in $\calP(d_1,d_2)$ from left to right by $u_1,u_2,\dots,u_{d_1}$ and the vertical edges from bottom to top by $v_1,v_2,\dots,v_{d_2}$.

\begin{defn}
We define the function $\pi: \calP(d_1,d_2) \to [0,d_1+d_2)$ by $\rd(\edge) \coloneqq  d_2x - d_1y$ for each edge $\edge$, where $(x,y)$ is the location of the south or west vertex of $\edge$.   
\end{defn}

When $d_1$ and $d_2$ are coprime, the maximal Dyck path $\calP(d_1,d_2)$ corresponds to the lower Christoffel word of slope $d_2/d_1$; see \cite[Sec. 1.2]{BLRS} for further details.  The following properties follow from this identification.

\begin{rem}
The quantity $\rd(\edge)$ is the relative distance of the left (resp. bottom) vertex to the main diagonal among all vertices of $\calP(d_1,d_2)$.  In particular, we have $\rd(\edge) = \rd(\edge')$ if their distances to the diagonal are equal, and we have $\rd(\edge) = 0$ if the corresponding vertex is on the main diagonal.  Moreover, as in \cite[Lemma 1.3]{BLRS}, if $d_1$ and $d_2$ are coprime, then $\rd(\edge )$ is distinct for all edges of $\calP(d_1,d_2)$.  
\end{rem}

More specifically, for each edge $\alpha$ of the maximal Dyck path $\mathcal{P}(d_1,d_2)$, with lattice point $(x,y)$ as its western or southern vertex, the value
$\pi(\alpha)/d_1 = \frac{d_2 x - d_1 y}{d_1} = \frac{d_2}{d_1}x - y \geq 0$ since $(x,y)$ lies on or below the line of slope $\frac{d_2}{d_1}$.  The fractional value is the exact distance to the diagonal.  If we label the edges of $\mathcal{P}(d_1,d_2)$ in order from $(0,0)$ to $(d_1,d_2)$ as $\alpha_0,\alpha_1,\alpha_2, \dots, \alpha_{d_1+d_2-1}$, this allows us to interpret $\pi(\alpha_k)$ as the modular residue of $k\cdot d_2 \mod (d_1+d_2)$. 

\begin{exmp} 
Consider the maximal Dyck path $\calP(5,3)$ that goes through the sequence of lattice points $(0,0), (1,0), (2,0),(2,1),(3,1),(4,1),(4,2),(5,2),(5,3)$ using the edges given by $u_1, u_2,v_1,u_3,u_4,v_2,u_5,v_3$.  Then we get all distinct integer values in $[0,d_1+d_2)$, for $d_1=5$, $d_2=3$, as expected:

\begin{center}
$\pi(\alpha_0) = \pi(u_1) = 3\cdot 0 - 5\cdot 0 = 0 $
\hspace{5em} $\pi(\alpha_4) = \pi(u_4) = 3\cdot 3 - 5\cdot 1 = 4$

$\pi(\alpha_1) = \pi(u_2) = 3\cdot 1 - 5\cdot 0 = 3 $ \hspace{5em} $\pi(\alpha_5) = \pi(v_2) = 3\cdot 4 - 5\cdot 1 = 7$

$\pi(\alpha_2) = \pi(v_1) = 3\cdot 2 - 5\cdot 0 = 6$
\hspace{5em} $\pi(\alpha_6) = \pi(u_5) = 3\cdot 4 - 5\cdot 2 = 2$

$\pi(\alpha_3) = \pi(u_3) = 3\cdot 2 - 5\cdot 1 = 1$ \hspace{5em} $\pi(\alpha_7) = \pi(v_3) = 3\cdot 5 - 5\cdot 2 = 5$

\end{center}
\end{exmp}

\subsection{Definition of tight gradings}
Motivated by Lee--Schiffler \cite{LS}, 
Lee, Li, and Zelevinsky \cite{LLZ} introduced combinatorial objects called \emph{compatible pairs} to construct the \emph{greedy basis} for rank-$2$ cluster algebras, consisting of indecomposable positive elements including the cluster monomials. Rupel \cite{Rupgengreed, Rup2} extended this construction to the setting of \emph{generalized} rank-$2$ cluster algebras by defining \emph{compatible gradings}.

A function from the set of edges on $\mathcal{P}(d_1,d_2)$ to $\mathbb{Z}_{\ge0}$ is called a \emph{grading}. Given a grading $\omega: \mathcal{P}(d_1,d_2) \to \Z_{\ge 0}$ and a set of edges $C \subset \mathcal{P}(d_1,d_2)$, we set $\omega(C) \coloneqq \sum_{\edge \in C}\omega(\edge)$.

\begin{defn}\label{def: compatible grading}
Let $\calP$ be an NE path.  A grading $\omega: \calP \rightarrow \Z_{\geq 0}$ is called \emph{compatible} if for every pair of $u\in \calP_\east$ and $v\in \calP_\north$ with $\omega(u)\omega(v)>0$, there exists an edge $\edge$ along the subpath $\overrightarrow{uv}$ so that at least one of the following holds:
  \begin{equation}\label{0407df:comp}
   \aligned
&  \edge \in \calP_\north \setminus\{v\} \quad \text{and} \quad |\overrightarrow{u\edge}_\north|= \omega\left( \overrightarrow{u\edge}_\east\right);\\
 &  \edge \in \calP_\east \setminus\{u\} \quad \text{and} \quad |\overrightarrow{\edge v}_\east|=\omega\left( \overrightarrow{\edge v}_\north\right).
  \endaligned
  \end{equation}

\end{defn} 

We denote the set of compatible gradings on $\calP(d_1,d_2)$ with $\omega(\calP_\north) = \pp$, $\omega(\calP_\east) = \qq$, by $\CG(d_1,d_2,\pp,\qq)$.
In the compatible pairs setting of \cite{LLZ}, $\omega$ takes at most one distinct nonzero value $c$ on the horizontal edges and at most one distinct nonzero value $b$ on the vertical edges. This is the setting relevant for the cluster algebra $\mathcal{A}(b,c)$ and the cluster scattering diagram $\frakD_{(b,c)}$.
Thus we define $\CGbc(d_1,d_2,\ii,\jj)$ to be the subset of gradings $\omega$ in $\CG(d_1,d_2,b\ii,c\jj)$ such that $\omega$ is either $b$ (resp. $c$) or $0$ on each edge of $\calP_\north$ (resp. $\calP_\east$).

In their study of compatible pairs, Lee, Li, and Zelevinsky \cite{LLZ} introduced the notion of the ``shadow'' of a set of horizontal (or vertical) edges, which Rupel \cite{Rup2} extended to the setting of gradings.

\begin{defn}\label{def: shadow}
Fix an NE path $\calP$ and a grading $\omega \colon \calP \rightarrow \Z_{\geq 0}$.
For each edge $\edge$ in $\calP$, we define its \emph{shadow}, denoted by $\sh(\edge)$ or $\sh(\edge;\omega)$, as follows. 
\begin{itemize}
    \item If $\edge$ is horizontal, then its shadow is $\overrightarrow{\edge v}_\north$, where $v \in \calP_\north$ is chosen such that $|\overrightarrow{\edge v}_\north| = \omega(\overrightarrow{\edge v}_\east)$ and $\overrightarrow{\edge v}$ has minimal length under this condition.  If no such $v$ exists, let $\sh(\edge) = \calP_\north$.
    \item If $\edge$ is vertical, then its shadow is $\overrightarrow{u\edge}_\east$, where $u \in \calP_\east$ is chosen such that $|\overrightarrow{u\edge}_\east| = \omega(\overrightarrow{u\edge}_\north)$ and $\overrightarrow{u\edge}$ has minimal length under this condition.  If no such $u$ exists, let $\sh(\edge) = \calP_\east$.
\end{itemize}  

For $C \subset \calP$, let the \emph{shadow} of $C$ be $\sh(C)=\bigcup_{\edge \in C} \sh(\edge)$.  
\end{defn}

Let $S_\east = S_\east(\omega)$ denote the set of horizontal edges $u \in \calP_\east$ such that $\omega(u) > 0$.  Similarly, we define $S_\north$ to be the set of vertical edges with nonzero value in the grading. 

\begin{defn}[tight grading]\label{def: tight grading}
    A compatible grading $\omega$ in $\CG(d_1,d_2,\pp,\qq)$ is called 
    \emph{tight} if
    \begin{itemize}
        \item $S_\east\subseteq \sh(S_\north)$ or $S_\north\subseteq \sh(S_\east)$ and
        \item $\pp d_2 - \qq d_1 = \epsilon \gcd(\pp, \qq)$ for $\epsilon = 1$ or $-1$.
    \end{itemize}
\end{defn}

\begin{rem}\label{rem: pos neg tight grading}
The two bulleted conditions in \Cref{def: tight grading} each have two possibilities, but the choice of one determines the other.  In particular, given a tight grading, we have $S_\east \subseteq \sh(S_\north)$ if and only if $\epsilon = 1$, and otherwise we have $S_\north \subseteq \sh(S_\east)$ when $\epsilon = -1$ (see, for example, \cite[Lemma 5.23]{BLM2}).  
\end{rem}

We now explain how tight gradings are related to wall-functions. Returning to the scattering diagram $\frakD(P_1, P_2)$ at the beginning of \Cref{section: change of lattice} where
\[
    P_i = 1 + \sum_{k\geq 1} p_{i,k}x_i^k \in \Bbbk\llbracket x_i\rrbracket .
\]
A main result of \cite{BLM2} is the following integrality and positivity statement for $\frakD(P_1, P_2)$: any wall-function coefficient behaves as a polynomial in $\{p_{i,k}\mid i\in \{1,2\}, k\geq 1\}$ with nonnegative integer coefficients given by counts of tight gradings.

\begin{thm}[{\cite[Theorem 5.17 and (5.35)]{BLM2}}]\label{thm: universal coefficient}
\leavevmode
    \begin{enumerate}[font=\upshape]
        \item There are nonnegative integers $\lambda(\pt_1, \pt_2)$ for pairs of partitions $(\pt_1, \pt_2)$ such that for any positive integers $d$ and $e$ with $\gcd(d,e)=1$, the wall-function on $\mathbb{R}_{\leq 0}(d, e)$ is given by
        \[
            \fclu{d}{e} = 1 + \sum_{k\geq 1}\sum_{\pt_1\vdash kd}\sum_{\pt_2\vdash ke}\lambda(\pt_1, \pt_2)\prod_{i\in \pt_1}p_{1,i}\prod_{j\in \pt_2}p_{2,j}\cdot x_1^{kd}x_2^{ke}.
        \]
        \item For partitions $\pt_1\vdash \pp$ and $\pt_2\vdash \qq$, the number $\lambda(\pt_1, \pt_2)$ counts tight gradings $\omega$ on $\mathcal P(d_1, d_2)$ with $\pp d_2 - \qq d_1 = \pm \gcd(\pp,\qq)$, $d_1\geq \pp$, $d_2\geq \qq$ such that 
        \[
            \pt_1 = (\omega(v)\mid v\in S_\north(\omega))\quad \text{and} \quad \pt_2 = (\omega(u)\mid u\in S_\east(\omega)),
        \]
        independent of the choice of $(d_1, d_2)$.  Here, we reorder the non-zero elements of the tuples $(\omega(v)\mid v\in S_\north(\omega))$ and $(\omega(u)\mid u\in S_\east(\omega))$ in weakly decreasing order so that they are integer partitions.

    \end{enumerate}
\end{thm}

\begin{rem}\label{rem: tight grading count}
    In the cluster setting, \Cref{thm: universal coefficient} states that 
    $$\tau^{b,c}(i,j) = \left |\left\{\omega \in \CGbc(d_1,d_2,i,j) \mid \omega \text{ is a tight grading}\right\}\right |\,,$$
    whenever $d_1\geq bi$, $d_2\geq cj$, and $|bid_2 - cjd_1| = \gcd(bi,cj)$.
\end{rem}

\subsection{Structure of Tight Gradings}
We now prove several structural results about tight gradings, which are needed for the later combinatorial constructions.  

We begin by establishing that tight gradings have only one ``component'', in the sense that the horizontal and vertical shadows have no gaps.  

\begin{lem}\label{lem: tight grading consecutiveness}
 Let $\omega$ be a tight grading on $\calP = \calP(d_1,d_2)$ with $\omega(\calP_\north) = p$ and $\omega(\calP_\east) = q$. Then there exists a unique $u_{\max} \in \calP_{\east}$ and $v_{\max} \in \calP_{\north}$ such that
 \begin{itemize}
    \item $\overrightarrow{u_{\max}v_{\max}} = \sh(S_\east) \cup \sh(S_\north) \cup \{v_{\max}\}$ if $pd_2 - qd_1 > 0$, or 
     \item $\overrightarrow{u_{\max}v_{\max}} = \{u_{\max}\} \cup \sh(S_\east) \cup \sh(S_\north)$ if $pd_2 - qd_1 < 0$.
 \end{itemize}
\end{lem}
\begin{proof}
    Let $\gamma \coloneqq \gcd(p,q) = |pd_2 - qd_1|$.  Without loss of generality, assume $pd_2 - qd_1  < 0$.  The case where $pd_2 - qd_1 > 0$ is analogous as we can swap the roles of horizontal and vertical edges.  Pick $u^{(i)} \in S_\east$ for $i = 1,2,\dots,\ell$ such that $\sh(S_\east)$ is the disjoint union of the shadows $\sh(u^{(i)})$.  Note that this is possible since edge shadows are either nested or disjoint.  For each $i$, let $v^{(i)}$ be the unique vertical edge such that $\displaystyle  \overrightarrow{u^{(i)}v^{(i)}}_\north = \sh(u^{(i)})$. Let $\pi_i \coloneqq \pi\left(u^{(i)}\right)$, $q_i \coloneqq  \left|\overrightarrow{u^{(i)}v^{(i)}}_\north\right|$, and $p_i \coloneqq  \left|\overrightarrow{u^{(i)}v^{(i)}}_\east\right| - 1$.  

    By the definition of compatibility for the grading $\omega$,  we have $q_id_1 > p_id_2 + \pi_i \geq p_i d_2$.  Moreover, since $\sum_{i=1}^\ell q_i = q$ and $qd_1 - pd_2 = \gcd(p,q)$, we can choose a partition $(\eta_1,\dots,\eta_\ell) \vdash \gamma$ such that $p_i = \frac{\eta_i}{\gamma}p$ and $q_i = \frac{\eta_i}{\gamma}q$.  In particular, this implies $\pi_i < q_id_1 - p_id_2 = \eta_i$.   

    We now aim to show that $u_1 \in \overrightarrow{u^{(i)}v^{(i)}}$ for all $i$.  Consider the horizontal edge $e$ that is $\pi_i\frac{p}{\gamma}$ steps to the right of $u^{(i)}$ (continuing cyclically around $\calP$ if necessary).  We then have 
    $$\pi(e) = \pi_i +\frac{\pi_i p}{\gamma} d_2 - \frac{\pi_i q}{\gamma}d_1 = \pi_i + \pi_i\cdot \frac{pd_2-qd_1}{\gamma} = \pi_i - \pi_i  = 0\,.$$
    So indeed, since $\pi_i\frac{p}{\gamma} < p_i$ and $u_1$ is the unique edge with $\pi(e) = 0$, the edge $u_1$ is included in the path $\overrightarrow{u^{(i)}v^{(i)}}$.  

    As we assumed the paths $\overrightarrow{u^{(i)}v^{(i)}}$ were disjoint, there can be only one such path, so $\ell = 1$.  We can then let $u_{\max} = u^{(1)}$.  We can also see that $|\sh(u_{\max})| = p+1$, so we can let $v_{\max} = v^{(1)}$.  
\end{proof}

 Next, we show that the tight gradings can only be located at certain positions along the Dyck path; this result is developed pictorially in \Cref{picto_tight}.

\begin{lem}\label{lem: tight grading bounding labels}
 Let $\omega$ be a tight grading on $\calP = \calP(d_1,d_2)$ with $\omega(\calP_\north) = p$ and $\omega(\calP_\east) = q$, and set $\gamma = \gcd(\pp,\qq)$. Let $u_{\max}$ and $v_{\max}$ be as in \Cref{lem: tight grading consecutiveness}.  Let $v_{\min}$ be the first vertical edge to the right of $u_{\max}$, and let $u_{\min}$ be the first horizontal edge below $v_{\max}$. 
\begin{itemize}
    \item If $\pp d_2 - \qq d_1 > 0$, we have $d_1 \leq \rd(v_{\max}) < d_1 + \gamma$, $\rd(u_{\max}) =  \rd(v_{\max}) - \gamma$, and $\omega(v_{\min}) = 0$.  
    \item If $\pp d_2 - \qq d_1 < 0$, we have $0 \leq \rd(u_{\max}) < \gamma$, $\rd(v_{\max}) = \rd(u_{\max}) + d_1+d_2-\gamma$, and $\omega(u_{\min}) = 0$.
\end{itemize}
\end{lem}
\begin{proof}
By definition, since $\omega$ is a tight grading, either $S_\east \subset \sh(S_\north)$ or $S_\north \subset \sh(S_\east)$ holds.   Assume $S_\east \subset \sh(S_\north)$, so by \Cref{rem: pos neg tight grading}, we are in the case where $\pp d_2 - \qq d_1 > 0$. 

The subpath $\mathfrak{s} = \overrightarrow{u_{\max}v_{\max}}$ contains all edges in $\sh(u_{\max})$, and its set of horizontal edges is precisely $\sh(v_{\max})$.  Moreover, note that $v_{\max} \notin \sh(S_\east)$.  Therefore $\mathfrak{s}$ contains exactly $p$ horizontal edges and contains at least $\qq + 1$ vertical edges.  Hence we have 
\begin{equation}\label{eqn: vertex label ineq}
  \pi(v_{\max}) - \pi(u_{\max}) =   |\mathfrak{s}_{\east}| d_2 -  (|\mathfrak{s}_{\north}| - 1) d_1 = \pp d_2  -  (|\mathfrak{s}_{\north}| - 1) d_1 \leq \pp d_2  -  \qq d_1 = \gamma\,.   
\end{equation}

As the edge $v_{\max}$ is vertical, we have $\pi(v_{\max}) \geq d_1$.  So 
$\pi(u_{\max}) \geq d_1 - \gamma$.  Moreover, since $\pi(v_{\max}) > \pi(u_{\max})$, we have 
$$p d_2 - (|\mathfrak{s}_{\north}| - 1) d_1 = \gamma - (|\mathfrak{s}_\north| - (q+1))d_1 > 0\,.$$ Then $\gamma \leq d_1$ implies that $|\mathfrak{s}_\north| = q+1$.  So the inequality in \Cref{eqn: vertex label ineq} is actually an equality, hence $\pi(u_{\max}) = \pi(v_{\max}) - \gamma$.  

Since $\pi(u_{\max}) \geq d_1 - \gamma \geq d_1 - d_2$, the edge $e$ immediately to the right of $u_{\max}$ must be vertical, as $\pi(e) \geq d_1$.  Hence $v_{\min}$ is immediately to the right of $u_{\max}$.  In order for $\omega$ to be compatible, we must have $\omega(v_{\min}) = 0$ since $\omega(u_{\max}) > 0$.  

The case $pd_2 - qd_1 < 0$ proceeds similarly, where we still consider the subpath $\mathfrak{s} = \overrightarrow{u_{\max}v_{\max}}$.  Now $\mathfrak{s}$ contains exactly $\qq$ vertical edges and at least $\pp + 1$ horizontal edges, so we have
\begin{equation}\label{eqn: vertex label ineq 2}
 \pi(v_{\max}) - \pi(u_{\max}) =  |\mathfrak{s}_\east| d_2 - (\qq - 1)d_1 \geq (\pp + 1)d_2 - (\qq - 1) d_1 =  d_1 + d_2 - \gamma\,.
\end{equation}

Since $\pi(v_{\max}) < d_1 + d_2$, we have $\pi(u_{\max}) < \gamma$.  As before, \Cref{eqn: vertex label ineq 2} also implies that $|\mathfrak{s}_\east| = p+1$, so $\pi(v_{\max}) = \pi(u_{\max}) + d_1 + d_2 - \gamma$.  The edge $e$ immediately preceding $v_{\max}$ has $\pi(e) = d_1 - \gamma < d_1$, so $e = u_{\min}$.  We then have $\omega(u_{\min}) = 0$ since $\omega(v_{\max}) > 0$. 
\end{proof}

The fact that certain edges must have value $0$ in the tight grading is needed later to demonstrate the Weyl symmetry in terms of tight gradings (see \Cref{section: weyl symmetry}).  

\begin{cor}\label{cor: length one partitions}
If $\pt_1 = (p)$ and $\pt_2 = (q)$, then $\lambda(\pt_1,\pt_2) = \gcd(p,q)$.  The corresponding tight gradings are such that the unique edges $u \in S_\east$ and $v \in S_\north$ satisfy $\pi(u) = \ell$ and $\pi(v) = d_1 + d_2 - \gcd(p,q) + \pi(u)$ for any $0 \leq \ell < \gcd(p,q)$.  \end{cor}
\begin{proof}
This follows directly from \Cref{lem: tight grading bounding labels} by noting that $u_{\max}$ and $v_{\max}$ are the only edges with nonzero value in any such tight grading.
\end{proof}

In the next section, we describe a way to visualize tight gradings that captures these structural results.

\section{A pictorial description of tight gradings via tiling }\label{picto_tight}

In this section, we give a pictorial description of tight gradings as partial tilings bounded by the underlying Dyck path and a rectangular \emph{frame}.  This serves not only as a helpful tool in constructing tight gradings, but also allows for comparisons of the relative ``tightness'' of gradings.

Throughout this section, we fix $(d_1,d_2,\pp,\qq)$ such that $\gamma = |\pp d_2 - \qq d_1| = \gcd(\pp,\qq)$.

\subsection{Frames}\label{subsec: frames}

First we consider the maximal Dyck path $\calP(2d_1,2d_2)$.  By considering the doubled Dyck path, we eliminate the need to consider the cyclic nature of the compatibility conditions when constructing tight gradings.  Thus we consider $\calP(2d_1,2d_2)$ as an ``unwrapping'' of $\calP(d_1,d_2)$.   
\begin{exmp}
    As  a running example, let $(d_1,d_2,\pp,\qq)=(14,9,12,8)$. 
The maximal Dyck path $\calP(2d_1,2d_2)=\calP(28,18)$ is given in \Cref{fig: doubled dyck path}, along with the four points in $P$.  Note that this is two copies of the Dyck path $\calP(14,9)$ glued end-to-end.
\end{exmp}

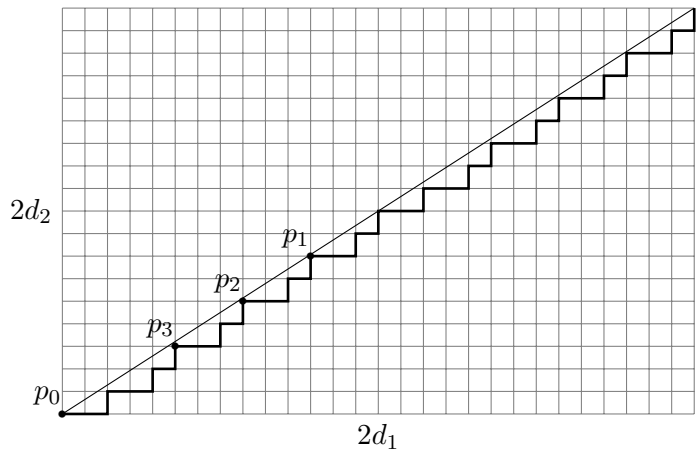
\begin{figure}
\begin{center}
\begin{tikzpicture}[scale=.3]
\draw[step=1,color=gray] (0,0) grid (2*14,2*9);
\draw (0,0)--(28,18);
\draw[line width=1, color=black] (0,0)--(2,0)--(2,1)--(4,1)--(4,2)--(5,2)--(5,3)--(7,3)--(7,4)--(8,4)--(8,5)--(10,5)--(10,6)--(11,6)--(11,7)--(13,7)--(13,8)--(14,8)--(14,9)--(14+2,9+0)--(14+2,9+1)--(14+4,9+1)--(14+4,9+2)--(14+5,9+2)--(14+5,9+3)--(14+7,9+3)--(14+7,9+4)--(14+8,9+4)--(14+8,9+5)--(14+10,9+5)--(14+10,9+6)--(14+11,9+6)--(14+11,9+7)--(14+13,9+7)--(14+13,9+8)--(14+14,9+8)--(14+14,9+9);
\draw (14,0) node[anchor=north]  {$2d_1$}; 
\draw (0,9) node[anchor=east]  {$2d_2$}; 
\draw (0,0) node {$\scriptscriptstyle\bullet$};
\draw (0,0)+(-1.7,0.8) node[anchor=west]  {$p_0$}; 
\draw (5,3) node {$\scriptscriptstyle\bullet$};
\draw (5,3)+(-1.7,0.8) node[anchor=west]  {$p_3$}; 
\draw (8,5) node {$\scriptscriptstyle\bullet$};
\draw (8,5)+(-1.7,0.8) node[anchor=west]  {$p_2$}; 
\draw (11,7) node {$\scriptscriptstyle\bullet$};
\draw (11,7)+(-1.7,0.8) node[anchor=west]  {$p_1$}; 
\end{tikzpicture}
\end{center}
\caption{$\mathcal{P}(2d_1, 2d_2)$}
\label{fig: doubled dyck path}
\end{figure}

Let $P=\{p_0,p_1,...,p_{\gamma-1}\}$ be the set of the $\gamma$ points on the first half of $\mathcal{P}(2d_1, 2d_2)$ that are closest to the main diagonal.  We label these such that $\pi(e_m) = m < \gamma$ where $e_m$ is the horizontal edge immediately to the right of $p_m$. Note that $p_0=(0,0)$ lies on the main diagonal, but none of the other $(\gamma-1)$ points in $P$ lie on the main diagonal.  In fact, the modular description of the Dyck path allows us to explicitly calculate the positions of the vertices $p_m$.  For $1 \leq m < \gamma$, we have
$$p_m = \left(d_1 - m\cdot \frac{p}{\gamma}, d_2 - m \cdot \frac{q}{\gamma}\right)\,.$$

\begin{rem}\label{rem: frame of tight}
    By \Cref{lem: tight grading consecutiveness} and \Cref{lem: tight grading bounding labels}, any tight grading on $\calP(d_1,d_2)$ is such that the left vertex of $u_{\max}$ (when $qd_1 > pd_2$) or the top vertex of $v_{\max}$ (when $pd_2 > qd_1$) is among the vertices in $P$.
\end{rem}

Fix some $\mathfrak{p}\in P$.  When $\qq d_1 > pd_2$, let $L_\mathfrak{p}$ be  the $(\pp +1)\times \qq$ rectangular frame whose southwestern vertex is at $\mathfrak{p}$. When $\pp d_2 > \qq d_1$, let $L_\mathfrak{p}$ be  the $\pp \times (\qq +1)$ rectangular frame whose northeastern vertex is at $\mathfrak{p}$.  Let $\mathcal{D}_\mathfrak{p}$ be the portion of the path $\calP(2d_1,2d_2)$ contained in $L_\mathfrak{p}$.

\begin{exmp}\label{pic_of_frame}
When $\mathfrak{p}= p_2 = (8,5)\in P$, the frame $L_\mathfrak{p}$ is given in \Cref{fig: frame}. 
\end{exmp}

\begin{figure}
\begin{center}
\begin{tikzpicture}[scale=.3]
\draw[step=1,color=gray] (0,0) grid (2*14,2*9);
\draw (0,0)--(28,18);
\draw[line width=1,color=black] (0,0)--(2,0)--(2,1)--(4,1)--(4,2)--(5,2)--(5,3)--(7,3)--(7,4)--(8,4)--(8,5)--(10,5)--(10,6)--(11,6)--(11,7)--(13,7)--(13,8)--(14,8)--(14,9)--(14+2,9+0)--(14+2,9+1)--(14+4,9+1)--(14+4,9+2)--(14+5,9+2)--(14+5,9+3)--(14+7,9+3)--(14+7,9+4)--(14+8,9+4)--(14+8,9+5)--(14+10,9+5)--(14+10,9+6)--(14+11,9+6)--(14+11,9+7)--(14+13,9+7)--(14+13,9+8)--(14+14,9+8)--(14+14,9+9);
\draw[line width=1.2, color=purple] (8,5)--(21,5)--(8+13,5+8)--(8,13)--(8,5);
\draw (6.7,5.5) node[anchor=west]  {$\mathfrak{p}$}; 
\draw (8,5) node {$\scriptscriptstyle\bullet$};
\draw (14,0) node[anchor=north]  {$2d_1$}; 
\draw (0,9) node[anchor=east]  {$2d_2$}; 
\end{tikzpicture}
\end{center}
\caption{A frame $L_\mathfrak{p}$.}
\label{fig: frame}
\end{figure}

\subsection{Construction of footprints associated to  $u_r^{(s)}$}
Let $\omega$ be a grading on $\calP(d_1,d_2)$ with $\omega(S_\north) = p$ and $\omega(S_\east) = q$. Let $\omega_{\mathfrak{p}}$ be the grading on $\mathcal{D}_{\mathfrak{p}}$ obtained by assigning an edge $e \in \calP(2d_1,2d_2)$ the value of $\omega$ on the unique edge $e'$ of $\calP(d_1,d_2)$ with $\pi(e') = \pi(e)$, then restricting to $\mathcal{D}_{\mathfrak{p}} \subset \calP(2d_1,2d_2)$.

 Let $u_1^{(1)}, u_1^{(2)}, \dots, u_1^{(\GG)}$ be the set of horizontal edges $u$ in $\mathcal{D}_\mathfrak{p}$ such that $\omega_{\mathfrak{p}}(\fph) > 0$, ordered from left-to-right.  For $s \in \{1,2,\dots,\GG\}$, let $\rho_s$ denote the number of horizontal edges strictly to the left of $u_1^{(s)}$ in $\mathcal{D}_\mathfrak{p}$.  We define $\mathcal{F}_{\east} (L_\mathfrak{p},\omega)$, which is a set of unit horizontal edges inside $L_{\mathfrak{p}} \cap (\R \times \Z)$ given by $$\mathcal{F}_{\east}(L_\mathfrak{p},\omega)=\{\fph_\fpi^{(\fpj)}\, :\, s \in\{1,2,...,\GG\} \text{ and }\fpi\in\{1,2,..., \rho_s + 1\}\}\,.$$
The \emph{footprint} (or \emph{tile}) of each $\fph_\fpi^{(\fpj)}$ is defined to be the $1\times \omega(\fph_1^{(\fpj)})$ rectangle whose south edge is  $\fph_\fpi^{(\fpj)}$. For $r > 1$, the edge $u_r^{(s)}$ is defined via the following double inductive process:

\noindent $\bullet$ The base case: By \Cref{lem: tight grading bounding labels}, we have $\rho_1 = 0$.  Thus, the base case of $r = 1$ consists of the construction of $u_1^{(1)}$ as above.

\noindent $\bullet$ The outer inductive hypothesis: suppose that $\fph_1^{(s)}, \fph_2^{(s)}, \dots, \fph_{\rho_s + 1}^{(s)}$ are constructed for some $s \in\{1,2,\dots,\GG\}$. 

\noindent $\bullet$ The inner inductive hypothesis: Suppose that $\fph_{1}^{(s+1)}, \fph_{2}^{(s+1)}, \dots, \fph_{\ell}^{(s+1)}$ are constructed for some $\ell \in\{1,2,\dots,\rho_{s+1} + 1\}$.

\noindent $\bullet$ The base step of the inner induction : construct $\fph_1^{(s+1)}$, as described above.

\noindent $\bullet$ The inductive step of the inner induction : $\fph_{\fpi+1}^{(s+1)}$ is uniquely determined by the following conditions:
\begin{enumerate}
    \item $\fph_{\fpi+1}^{(s+1)}$ in the column directly to the left of $\fph_{\fpi}^{(s+1)}$;
    \item Either 
    \begin{enumerate}
        \item  $\fph_{\fpi+1}^{(s+1)}$ is Northwest of $\fph_{\fpi}^{(s+1)}$, in which case its footprint is colored cyan; or  
        \item $\fph_{\fpi+1}^{(s+1)}$ is directly to the West of $\fph_{\fpi}^{(s+1)}$, in which case its footprint is colored dark blue; 
    \end{enumerate}
    \item and $\fph_{\fpi+1}^{(s+1)}$ is as far South as possible, while the interior of the footprint of $\fph_{\fpi+1}^{(s+1)}$ does not overlap with the footprint of $\fph_{r'}^{(s')}$ for $r' \leq  r$ or $s' \leq s$.
\end{enumerate}

\begin{defn}
    We define $\text{DarkBlue}(w)$ as the union of the interior of dark blue tiles, and $\text{Cyan}(w)$  as the union of the interior of cyan tiles.
\end{defn}

\begin{exmp}\label{ex_of_blue_green}
Continue from \Cref{pic_of_frame} above. Consider a grading $\omega$ on $\calP(d_1,d_2)$ such that $S_\mathbf{E}(\omega_{\mathfrak{p}})=\{\fph_1^{(1)},\fph_1^{(2)},\fph_1^{(3)},\fph_1^{(4)} \}$ and $\omega_{\mathfrak{p}}(\fph_1^{(\fpj)})=2$ for $\fpj\in\{1,2,3,4\}$, where $\fph_1^{(1)}$ is the horizontal line segment joining $(8,5)$ and $(9,5)$, $\fph_1^{(2)}$ joining $(9,5)$ and $(10,5)$, $\fph_1^{(3)}$ joining $(10,6)$ and $(11,6)$, and $\fph_1^{(4)}$ joining $(14,9)$ and $(15,9)$. Then
$$\aligned
&\fph_2^{(2)} \text{ joins }(8,7)\text{ and }(9,7);\\
&\\
&\fph_2^{(3)} \text{ joins }(9,7)\text{ and }(10,7);\\
&\fph_3^{(3)} \text{ joins }(8,9)\text{ and }(9,9);\\
&\\
&\fph_\ell^{(4)} \text{ joins }(15-\ell,9)\text{ and }(16-\ell,9) \text{ for } \ell = 2,3,4,5,6;\\
&\fph_7^{(4)} \text{ joins }(8,11)\text{ and }(9,11).\\
\endaligned$$
The following picture shows the interior of the footprint of each $\fph_{\fpi}^{(\fpj)}$. 
\end{exmp}

\begin{center}
\begin{tikzpicture}[scale=.4]
\draw[step=1,color=gray] (0,0) grid (13,8);
\draw[line width=1.2, color=purple] (8-8,5-5)--(21-8,5-5)--(13,8)--(0,8)--(0,0);
\draw[line width=1.5,color=black] (0,0)--(2,0)--(2,1)--(3,1)--(3,2)--(5,2)--(5,3)--(6,3)--(6,4)--(8,4)--(8,5)--(10,5)--(10,6)--(11,6)--(11,7)--(13,7)--(13,8);
\draw (0.1,-0.4) node[anchor=east]  {$\mathfrak{p}$}; 
\fill[cyan] (1.2,0.2)--(1.2,1.8)--(1.8,1.8)--(1.8,0.2)--(1.2,0.2);
\fill[cyan] (1.2,2+0.2)--(1.2,2+1.8)--(1.8,2+1.8)--(1.8,2+0.2)--(1.2,2+0.2);
\fill[cyan] (1+1.2,1+0.2)--(1+1.2,1+1.8)--(1+1.8,1+1.8)--(1+1.8,1+0.2)--(1+1.2,1+0.2);
\fill[cyan] (2+1.2+3-6 ,4+0.2-4)--(5+1.2-6,4+1.8-4)--(5+1.8-6,4+1.8-4)--(5+1.8-6,4+0.2-4)--(5+1.2-6,4+0.2-4);
\fill[cyan] (2+1.2+3-6 ,4+0.2-2)--(5+1.2-6,4+1.8-2)--(5+1.8-6,4+1.8-2)--(5+1.8-6,4+0.2-2)--(5+1.2-6,4+0.2-2);
\fill[cyan] (2+1.2+3-6 ,4+0.2)--(5+1.2-6,4+1.8)--(5+1.8-6,4+1.8)--(5+1.8-6,4+0.2)--(5+1.2-6,4+0.2);
\fill[cyan] (0.2,0.2)--(0.2,1.8)--(0.8,1.8)--(0.8,0.2)--(0.2,0.2);
\fill[cyan] (0.2,2+0.2)--(0.2,2+1.8)--(0.8,2+1.8)--(0.8,2+0.2)--(0.2,2+0.2);
\fill[cyan] (0.2,4+0.2)--(0.2,4+1.8)--(0.8,4+1.8)--(0.8,4+0.2)--(0.2,4+0.2);
\fill[cyan] (0.2,6+0.2)--(0.2,6+1.8)--(0.8,6+1.8)--(0.8,6+0.2)--(0.2,6+0.2);
\fill[cyan] (1.2,0.2)--(1.2,1.8)--(1.8,1.8)--(1.8,0.2)--(1.2,0.2);
\fill[cyan] (1.2,2+0.2)--(1.2,2+1.8)--(1.8,2+1.8)--(1.8,2+0.2)--(1.2,2+0.2);
\fill[blue] (1.2,4+0.2)--(1.2,4+1.8)--(1.8,4+1.8)--(1.8,4+0.2)--(1.2,4+0.2);
\fill[cyan] (1+1.2,1+0.2)--(1+1.2,1+1.8)--(1+1.8,1+1.8)--(1+1.8,1+0.2)--(1+1.2,1+0.2);
\fill[blue] (1+1.2,4+0.2)--(1+1.2,4+1.8)--(1+1.8,4+1.8)--(1+1.8,4+0.2)--(1+1.2,4+0.2);
\fill[blue] (3+1.2,4+0.2)--(3+1.2,4+1.8)--(3+1.8,4+1.8)--(3+1.8,4+0.2)--(3+1.2,4+0.2);
\fill[blue] (2+1.2,4+0.2)--(2+1.2,4+1.8)--(2+1.8,4+1.8)--(2+1.8,4+0.2)--(2+1.2,4+0.2);
\fill[blue] (4+1.2,4+0.2)--(4+1.2,4+1.8)--(4+1.8,4+1.8)--(4+1.8,4+0.2)--(4+1.2,4+0.2);
\fill[cyan] (2+1.2+3 ,4+0.2)--(5+1.2,4+1.8)--(5+1.8,4+1.8)--(5+1.8,4+0.2)--(5+1.2,4+0.2);
\end{tikzpicture}
\end{center}

\subsection{Construction of footprints associated to  $v_{\fpi}^{(\fpj)}$}
Let $v_1^{(1)}, v_1^{(2)}, \dots, v_1^{(\GG')}$ be the set of vertical edges $u$ in $\mathcal{D}_\mathfrak{p}$ such that $\omega_{\mathfrak{p}}(\fph) > 0$, ordered from top-to-bottom.  For $s \in \{1,2,\dots,\GG\}$, let $\rho'_s$ denote the number of vertical edges strictly to the above of $v_1^{(s)}$ in $\mathcal{D}_\mathfrak{p}$.  We define $\mathcal{F}_{\east} (L_\mathfrak{p},\omega)$, which is a set of unit vertical edges inside $L_{\mathfrak{p}} \cap (\Z \times \R)$ given by $$\mathcal{F}_{\east}(L_\mathfrak{p},\omega)=\{\fpv_\fpi^{(\fpj)}\, :\, s \in\{1,2,...,\GG'\} \text{ and }\fpi\in\{1,2,..., \rho_s' + 1\}\}$$ in $L_\mathfrak{p}$. 
The \emph{footprint} (or \emph{tile}) of each $v_{\fpi}^{(j)}$ is defined to be the $\omega(v_1^{(\fpj)})\times 1$ rectangle whose east edge is  $v_{\fpi}^{(\fpj)}$. Construct $v_{\fpi}^{(\fpj)}$ in a similar fashion:
\begin{enumerate}
    \item $\fpv_{\fpi+1}^{(s+1)}$ in the row directly above $\fpv_{\fpi}^{(s+1)}$;
    \item Either 
    \begin{enumerate}
        \item  $\fpv_{\fpi+1}^{(s+1)}$ is Northwest of $\fpv_{\fpi}^{(s+1)}$, in which case its footprint is colored red; or
        \item $\fpv_{\fpi+1}^{(s+1)}$ is directly North of $\fpv_{\fpi}^{(s+1)}$, in which case its footprint is colored dark red; 
    \end{enumerate}
    \item and $\fpv_{\fpi+1}^{(s+1)}$ is as far East as possible, while the interior of the footprint of $\fpv_{\fpi+1}^{(s+1)}$ does not overlap with the footprint of $\fpv_{r'}^{(s')}$ for $r' \leq  r$ or $s' \leq s$.
\end{enumerate}

\begin{defn}
    We define $\text{Red}(w)$ as the union of the interior of red tiles, and $\text{DarkRed}(w)$  as the union of the interior of dark red tiles.
\end{defn}

\begin{exmp}\label{ex_of_red_orange}
Continue from Example~\ref{pic_of_frame}. Consider a grading $\omega$ on $\calP(d_1,d_2)$ where $S_\mathbf{N}(\omega_\mathfrak{p})=\{v_1^{(1)},v_1^{(2)},v_1^{(3)},v_1^{(4)} \}$ and $\omega_\mathfrak{p}(v_1^{(\fpj)})=3$ for $\fpj\in\{1,2,3,4\}$, where $v_1^{(1)}$ is the vertical line segment joining $(21,12)$ and $(21,13)$, $v_1^{(2)}$ joins $(19,11)$ and $(19,12)$, $v_1^{(3)}$ joins $(18,10)$ and $(18,11)$, and $v_1^{(4)}$ joins $(14,8)$ and $(14,9)$.

\begin{center}
\begin{tikzpicture}[scale=.4]
\draw[step=1,color=gray] (0,0) grid (13,8);
\draw[line width=1.2, color=purple] (8-8,5-5)--(21-8,5-5)--(13,8)--(0,8)--(0,0);
\draw[line width=1.5,color=black] (0,0)--(2,0)--(2,1)--(3,1)--(3,2)--(5,2)--(5,3)--(6,3)--(6,4)--(8,4)--(8,5)--(10,5)--(10,6)--(11,6)--(11,7)--(13,7)--(13,8);
\draw (0.1,-0.4) node[anchor=east]  {$\mathfrak{p}$}; 
\fill[red] (3+0.2, 3+0.2)--(3+3-0.2, 3+0.2)--(3+3-0.2, 4-0.2)--(3+0.2, 4-0.2)--(3+0.2, 3+0.2);
\fill[maroon] (3+0.2, 3+1+0.2)--(3+3-0.2, 3+1+0.2)--(3+3-0.2, 4+1-0.2)--(3+0.2, 4+1-0.2)--(3+0.2, 3+1+0.2);
\fill[maroon] (3+0.2, 3+2+0.2)--(3+3-0.2, 3+2+0.2)--(3+3-0.2, 4+2-0.2)--(3+0.2, 4+2-0.2)--(3+0.2, 3+2+0.2);
\fill[red] (2+5+0.2, 1+4+0.2)--(2+5+3-0.2, 1+4+0.2)--(2+5+3-0.2, 1+4+1-0.2)--(2+5+0.2, 1+4+1-0.2)--(2+5+0.2, 1+4+0.2);
\fill[red] (-3+5+0.2, 2+4+0.2)--(-3+5+3-0.2, 2+4+0.2)--(-3+5+3-0.2, 2+4+1-0.2)--(-3+5+0.2, 2+4+1-0.2)--(-3+5+0.2, 2+4+0.2);
\fill[red] (5+0.2, 2+4+0.2)--(5+3-0.2, 2+4+0.2)--(5+3-0.2, 2+4+1-0.2)--(5+0.2, 2+4+1-0.2)--(5+0.2, 2+4+0.2);
\fill[red] (3+5+0.2, 2+4+0.2)--(3+5+3-0.2, 2+4+0.2)--(3+5+3-0.2, 2+4+1-0.2)--(3+5+0.2, 2+4+1-0.2)--(3+5+0.2, 2+4+0.2);
\fill[red] (-3-1+5+0.2, 2+1+4+0.2)--(-3-1+5+3-0.2, 2+1+4+0.2)--(-3-1+5+3-0.2, 2+1+4+1-0.2)--(-3-1+5+0.2, 2+1+4+1-0.2)--(-3-1+5+0.2, 2+1+4+0.2);
\fill[red] (-1+5+0.2, 2+1+4+0.2)--(-1+5+3-0.2, 2+1+4+0.2)--(-1+5+3-0.2, 2+1+4+1-0.2)--(-1+5+0.2, 2+1+4+1-0.2)--(-1+5+0.2, 2+1+4+0.2);
\fill[red] (3-1+5+0.2, 2+1+4+0.2)--(3-1+5+3-0.2, 2+1+4+0.2)--(3-1+5+3-0.2, 2+1+4+1-0.2)--(3-1+5+0.2, 2+1+4+1-0.2)--(3-1+5+0.2, 2+1+4+0.2);
\fill[red] (6-1+5+0.2, 2+1+4+0.2)--(6-1+5+3-0.2, 2+1+4+0.2)--(6-1+5+3-0.2, 2+1+4+1-0.2)--(6-1+5+0.2, 2+1+4+1-0.2)--(6-1+5+0.2, 2+1+4+0.2);
\end{tikzpicture}
\end{center}
\end{exmp}

\subsection{Reinterpretation of tight gradings}

We now describe tight gradings purely in terms of their footprints.  This gives a visual tool for constructing and classifying tight gradings.
\begin{prop}\label{redblue}
  The grading $\omega$ is tight if and only if there is some $\mathfrak{p} \in P$ such that $\omega_{\mathfrak{p}}(S_\north) = p$, $\omega_{\mathfrak{p}}(S_\east) = q$, and 
  $$\mathrm{Red}(w)\cap \mathrm{Cyan}(w)=\mathrm{Red}(w)\cap \mathrm{DarkBlue}(w)=\mathrm{DarkRed}(w)\cap \mathrm{Cyan}(w)=\varnothing.$$
\end{prop}

\begin{rem} Note that dark blue and dark red footprints are allowed to intersect.  We depict this intersection in purple. We combine the purple regions according to the following rule: the intersection of $u_{r_1}^{(s_1)}$ and $v_{r_2}^{(s_2)}$ is in the same purple component as the intersection of $u_{r_1'}^{(s_1')}$ and $v_{r_2'}^{(s_2')}$ if and only if $s_1 = s_1'$ and $s_2 = s_2'$.
\end{rem}

Before presenting the proof of \Cref{redblue}, we illustrate the construction in some examples.

\begin{exmp}
Continue from Examples~\ref{ex_of_blue_green} and \ref{ex_of_red_orange}. See the picture below. Since $$\mathrm{Red}(w)\cap \mathrm{Cyan}(w)=\mathrm{Red}(w)\cap \mathrm{DarkBlue}(w)=\mathrm{DarkRed}(w)\cap \mathrm{Cyan}(w)=\varnothing,$$ the grading $\omega$ is tight.

\begin{center}
\begin{tikzpicture}[scale=.4]
\draw[step=1,color=gray] (0,0) grid (13,8);
\draw[line width=1.2, color=purple] (8-8,5-5)--(21-8,5-5)--(13,8)--(0,8)--(0,0);
\draw[line width=1.5,color=black] (0,0)--(2,0)--(2,1)--(3,1)--(3,2)--(5,2)--(5,3)--(6,3)--(6,4)--(8,4)--(8,5)--(10,5)--(10,6)--(11,6)--(11,7)--(13,7)--(13,8);
\draw (0.1,-0.4) node[anchor=east]  {$\mathfrak{p}$}; 
\fill[cyan] (1.2,0.2)--(1.2,1.8)--(1.8,1.8)--(1.8,0.2)--(1.2,0.2);
\fill[cyan] (1.2,2+0.2)--(1.2,2+1.8)--(1.8,2+1.8)--(1.8,2+0.2)--(1.2,2+0.2);
\fill[cyan] (1+1.2,1+0.2)--(1+1.2,1+1.8)--(1+1.8,1+1.8)--(1+1.8,1+0.2)--(1+1.2,1+0.2);
\fill[cyan] (2+1.2+3-6 ,4+0.2-4)--(5+1.2-6,4+1.8-4)--(5+1.8-6,4+1.8-4)--(5+1.8-6,4+0.2-4)--(5+1.2-6,4+0.2-4);
\fill[cyan] (2+1.2+3-6 ,4+0.2-2)--(5+1.2-6,4+1.8-2)--(5+1.8-6,4+1.8-2)--(5+1.8-6,4+0.2-2)--(5+1.2-6,4+0.2-2);
\fill[cyan] (2+1.2+3-6 ,4+0.2)--(5+1.2-6,4+1.8)--(5+1.8-6,4+1.8)--(5+1.8-6,4+0.2)--(5+1.2-6,4+0.2);
\fill[cyan] (0.2,0.2)--(0.2,1.8)--(0.8,1.8)--(0.8,0.2)--(0.2,0.2);
\fill[cyan] (0.2,2+0.2)--(0.2,2+1.8)--(0.8,2+1.8)--(0.8,2+0.2)--(0.2,2+0.2);
\fill[cyan] (0.2,4+0.2)--(0.2,4+1.8)--(0.8,4+1.8)--(0.8,4+0.2)--(0.2,4+0.2);
\fill[cyan] (0.2,6+0.2)--(0.2,6+1.8)--(0.8,6+1.8)--(0.8,6+0.2)--(0.2,6+0.2);
\fill[cyan] (1.2,0.2)--(1.2,1.8)--(1.8,1.8)--(1.8,0.2)--(1.2,0.2);
\fill[cyan] (1.2,2+0.2)--(1.2,2+1.8)--(1.8,2+1.8)--(1.8,2+0.2)--(1.2,2+0.2);
\fill[blue] (1.2,4+0.2)--(1.2,4+1.8)--(1.8,4+1.8)--(1.8,4+0.2)--(1.2,4+0.2);
\fill[cyan] (1+1.2,1+0.2)--(1+1.2,1+1.8)--(1+1.8,1+1.8)--(1+1.8,1+0.2)--(1+1.2,1+0.2);
\fill[blue] (1+1.2,4+0.2)--(1+1.2,4+1.8)--(1+1.8,4+1.8)--(1+1.8,4+0.2)--(1+1.2,4+0.2);
\fill[cyan] (2+1.2+3 ,4+0.2)--(5+1.2,4+1.8)--(5+1.8,4+1.8)--(5+1.8,4+0.2)--(5+1.2,4+0.2);
\fill[red] (3+0.2, 3+0.2)--(3+3-0.2, 3+0.2)--(3+3-0.2, 4-0.2)--(3+0.2, 4-0.2)--(3+0.2, 3+0.2);

\fill[violet] (3+0.2, 3+1+0.2)--(3+3-0.2, 3+1+0.2)--(3+3-0.2, 5+1-0.2)--(3+0.2, 5+1-0.2)--(3+0.2, 3+1+0.2);

\fill[red] (2+5+0.2, 1+4+0.2)--(2+5+3-0.2, 1+4+0.2)--(2+5+3-0.2, 1+4+1-0.2)--(2+5+0.2, 1+4+1-0.2)--(2+5+0.2, 1+4+0.2);
\fill[red] (-3+5+0.2, 2+4+0.2)--(-3+5+3-0.2, 2+4+0.2)--(-3+5+3-0.2, 2+4+1-0.2)--(-3+5+0.2, 2+4+1-0.2)--(-3+5+0.2, 2+4+0.2);
\fill[red] (5+0.2, 2+4+0.2)--(5+3-0.2, 2+4+0.2)--(5+3-0.2, 2+4+1-0.2)--(5+0.2, 2+4+1-0.2)--(5+0.2, 2+4+0.2);
\fill[red] (3+5+0.2, 2+4+0.2)--(3+5+3-0.2, 2+4+0.2)--(3+5+3-0.2, 2+4+1-0.2)--(3+5+0.2, 2+4+1-0.2)--(3+5+0.2, 2+4+0.2);
\fill[red] (-3-1+5+0.2, 2+1+4+0.2)--(-3-1+5+3-0.2, 2+1+4+0.2)--(-3-1+5+3-0.2, 2+1+4+1-0.2)--(-3-1+5+0.2, 2+1+4+1-0.2)--(-3-1+5+0.2, 2+1+4+0.2);
\fill[red] (-1+5+0.2, 2+1+4+0.2)--(-1+5+3-0.2, 2+1+4+0.2)--(-1+5+3-0.2, 2+1+4+1-0.2)--(-1+5+0.2, 2+1+4+1-0.2)--(-1+5+0.2, 2+1+4+0.2);
\fill[red] (3-1+5+0.2, 2+1+4+0.2)--(3-1+5+3-0.2, 2+1+4+0.2)--(3-1+5+3-0.2, 2+1+4+1-0.2)--(3-1+5+0.2, 2+1+4+1-0.2)--(3-1+5+0.2, 2+1+4+0.2);
\fill[red] (6-1+5+0.2, 2+1+4+0.2)--(6-1+5+3-0.2, 2+1+4+0.2)--(6-1+5+3-0.2, 2+1+4+1-0.2)--(6-1+5+0.2, 2+1+4+1-0.2)--(6-1+5+0.2, 2+1+4+0.2);
\end{tikzpicture}
\end{center}

\end{exmp}

\begin{exmp}
Consider another grading $\omega$ on  $\calP(d_1,d_2)$.  Let $S_\east$ consist of the horizontal line  segment $\fph_1^{(1)}$  joining $(8,5)$ and $(9,5)$, $\fph_1^{(2)}$ joining $(9,5)$ and $(10,5)$, $\fph_1^{(3)}$ joining $(10,6)$ and $(11,6)$, $\fph_1^{(4)}$ joining $(11,7)$ and $(12,7)$, and $\fph_1^{(4)}$ joining $(14,9)$ and $(15,9)$. Let $S_\north$ consist of the vertical line segment $v_1^{(1)}$ joining $(21,12)$ and $(21,13)$, $v_1^{(2)}$ joining $(19,11)$ and $(19,12)$, $v_1^{(3)}$ joining $(18,10)$ and $(18,11)$, $v_1^{(4)}$ joining $(14,8)$ and $(14,9)$, and $v_1^{(5)}$ joining $(13,7)$ and $(13,8)$.  We then set 

$$\omega(e) = \begin{cases} 1 & \text{ if } e = \fph_1^{(3)},\; \fph_1^{(4)}, \;v_1^{(5)};\\
2 & \text{ if } e = \fph_1^{(1)}, \;\fph_1^{(2)},\; \fph_1^{(5)}, \;\fpv_1^{(2)};\\
3 & \text{ if } e = \fpv_1^{(1)}, \;\fpv_1^{(3)},\; \fpv_1^{(4)}\,.\\
\end{cases}$$
The resulting tiling is shown below.

\begin{center}
\begin{tikzpicture}[scale=.4]
\draw[step=1,color=gray] (0,0) grid (13,8);
\draw[line width=1.2, color=purple] (8-8,5-5)--(21-8,5-5)--(13,8)--(0,8)--(0,0);
\draw[line width=1.5,color=black] (0,0)--(2,0)--(2,1)--(3,1)--(3,2)--(5,2)--(5,3)--(6,3)--(6,4)--(8,4)--(8,5)--(10,5)--(10,6)--(11,6)--(11,7)--(13,7)--(13,8);
\draw (0.1,-0.4) node[anchor=east]  {$\mathfrak{p}$}; 
\fill[cyan] (1.2,0.2)--(1.2,1.8)--(1.8,1.8)--(1.8,0.2)--(1.2,0.2);
\fill[cyan] (0.2,0.2)--(0.2,1.8)--(0.8,1.8)--(0.8,0.2)--(0.2,0.2);
\fill[cyan] (0.2,2+0.2)--(0.2,2+1.8)--(0.8,2+1.8)--(0.8,2+0.2)--(0.2,2+0.2);
\fill[cyan] (0.2,4+0.2)--(0.2,4+0.8)--(0.8,4+0.8)--(0.8,4+0.2)--(0.2,4.2);
\fill[cyan] (0.2,4+1+0.2)--(0.2,4+1.8)--(0.8,4+1.8)--(0.8,5+0.2)--(0.2,1+4.2);
\fill[cyan] (0.2,6+0.2)--(0.2,6+1.8)--(0.8,6+1.8)--(0.8,6+0.2)--(0.2,6+0.2);
\fill[cyan] (1.2,0.2)--(1.2,1.8)--(1.8,1.8)--(1.8,0.2)--(1.2,0.2);
\fill[cyan] (1.2,2+0.2)--(1.2,2+0.8)--(1.8,2+0.8)--(1.8,2+0.2)--(1.2,2.2);
\fill[cyan] (1.2,2+1+0.2)--(1.2,2+1.8)--(1.8,2+1.8)--(1.8,3+0.2)--(1.2,1+2.2);
\fill[blue] (1.2,4+0.2)--(1.2,4+1.8)--(1.8,4+1.8)--(1.8,4+0.2)--(1.2,4+0.2);
\fill[cyan] (1+1.2,1+0.2)--(1+1.2,1.8)--(1+1.8,1.8)--(1+1.8,1+0.2)--(1+1.2,1+0.2);
\fill[blue] (1+1.2,1+1+0.2)--(1+1.2,1+1.8)--(1+1.8,1+1.8)--(1+1.8,1+1+0.2)--(1+1.2,1+1+0.2);
\fill[cyan] (2+1.2,1+1+0.2)--(2+1.2,1+1.8)--(2+1.8,1+1.8)--(2+1.8,1+1+0.2)--(2+1.2,1+1+0.2);
\fill[cyan] (2+1.2+3 ,4+0.2)--(5+1.2,4+1.8)--(5+1.8,4+1.8)--(5+1.8,4+0.2)--(5+1.2,4+0.2);
\fill[red] (3+0.2, 3+0.2)--(3+3-0.2, 3+0.2)--(3+3-0.2, 4-0.2)--(3+0.2, 4-0.2)--(3+0.2, 3+0.2);
\fill[red] (4+0.2, 2+0.2)--(5-0.2, 2+0.2)--(5-0.2, 3-0.2)--(4+0.2, 3-0.2)--(4+0.2, 2+0.2);
\fill[red] (4+0.2-2, 2+0.2+1)--(5-2-0.2, 2+1+0.2)--(5-2-0.2, 3+1-0.2)--(4-2+0.2, 3+1-0.2)--(4-2+0.2, 2+1+0.2);
\fill[violet] (3+0.2, 3+1+0.2)--(3+3-0.2, 3+1+0.2)--(3+3-0.2, 5+1-0.2)--(3+0.2, 5+1-0.2)--(3+0.2, 3+1+0.2);
\fill[violet] (1+1.2,4+0.2)--(1+1.2,4+1.8)--(1+1.8,4+1.8)--(1+1.8,4+0.2)--(1+1.2,4+0.2);
\fill[red] (2+5+0.2, 1+4+0.2)--(2+5+3-0.2, 1+4+0.2)--(2+5+3-0.2, 1+4+1-0.2)--(2+5+0.2, 1+4+1-0.2)--(2+5+0.2, 1+4+0.2);
\fill[maroon] (3+0.2, 6+0.2)--(3+3-0.2, 6+0.2)--(3+3-0.2, 7-0.2)--(3+0.2, 7-0.2)--(3+0.2, 6+0.2);
\fill[maroon] (3-1+0.2, 6+0.2)--(3-0.2, 6+0.2)--(3-0.2, 7-0.2)--(3-1+0.2, 7-0.2)--(3-1+0.2, 6+0.2);
\fill[red] (6+0.2, 2+4+0.2)--(3+5+1-0.2, 2+4+0.2)--(3+5+1-0.2, 2+4+1-0.2)--(3+3+0.2, 2+4+1-0.2)--(3+3+0.2, 2+4+0.2);
\fill[red] (9+0.2, 2+4+0.2)--(3+7+1-0.2, 2+4+0.2)--(3+7+1-0.2, 2+4+1-0.2)--(9+0.2, 2+4+1-0.2)--(9+0.2, 2+4+0.2);
\fill[red] (1+0.2, 2+1+4+0.2)--(2-0.2, 2+1+4+0.2)--(2-0.2, 2+1+4+1-0.2)--(1+0.2, 2+1+4+1-0.2)--(1+0.2, 2+1+4+0.2);
\fill[red] (-3+5+0.2, 2+1+4+0.2)--(-3+5+3-0.2, 2+1+4+0.2)--(-3+5+3-0.2, 2+1+4+1-0.2)--(-3+5+0.2, 2+1+4+1-0.2)--(-3+5+0.2, 2+1+4+0.2);
\fill[red] (5+0.2, 2+1+4+0.2)--(5+3-0.2, 2+1+4+0.2)--(5+3-0.2, 2+1+4+1-0.2)--(5+0.2, 2+1+4+1-0.2)--(5+0.2, 2+1+4+0.2);
\fill[red] (3+5+0.2, 2+1+4+0.2)--(3-1+5+3-0.2, 2+1+4+0.2)--(3-1+5+3-0.2, 2+1+4+1-0.2)--(3+5+0.2, 2+1+4+1-0.2)--(3+5+0.2, 2+1+4+0.2);
\fill[red] (6-1+5+0.2, 2+1+4+0.2)--(6-1+5+3-0.2, 2+1+4+0.2)--(6-1+5+3-0.2, 2+1+4+1-0.2)--(6-1+5+0.2, 2+1+4+1-0.2)--(6-1+5+0.2, 2+1+4+0.2);
\end{tikzpicture}
\end{center}

\end{exmp}

Examples where neither dark blue nor dark red tiles appear are constructed in \Cref{ex: d=1 e=j}.

\begin{proof}[Proof of Proposition~\ref{redblue}]
Let $\omega$ be a grading and $\mathfrak{p} \in P$ such that $\omega_{\mathfrak{p}}(S_\north) = p$ and $\omega_{\mathfrak{p}}(S_\east) = q$.
Suppose that $\text{Red}(w)\cap \text{Cyan}(w)\neq\varnothing$. 

A rectangle in $L_{\mathfrak{p}}$ whose vertices are lattice points and whose edges are vertical or horizontal will be called a \emph{lattice rectangle}. A lattice rectangle whose southwestern corner and northeastern corner lie on the maximal Dyck path will be called a  \emph{Dyck rectangle}.
For a Dyck rectangle $R$ in $L_{\mathfrak{p}}$ whose northwestern vertex is at $(x_W, y_N)$, let $b_R$ be the interior of the unit square whose vertices are at $(x_W-1/2\pm1/2, y_N-1/2\pm1/2)$.

Let $T\subseteq L_{\mathfrak{p}}$ be a Dyck rectangle such that  $\text{Red}(w)\cap \text{Cyan}(w)\cap T=b_T$. 
Let $u$ be the horizontal edge whose left endpoint is at the southwestern corner of $T$. Let $v$ be the vertical edge whose upper endpoint is at the northeastern corner of $T$. 
Then for any  edge $e$ along the subpath $\overrightarrow{uv}$, none of the following holds:
  \begin{equation*}
   \aligned
&  e\in \calP_\north \setminus\{v\} \quad \text{and} \quad |\overrightarrow{ue}_\north|= \omega\left( \overrightarrow{ue}_\east\right);\\
 &  e\in \calP_\east \setminus\{u\} \quad \text{and} \quad |\overrightarrow{ev}_\east|=\omega\left( \overrightarrow{ev}_\north\right).
  \endaligned
  \end{equation*}
So $\omega$ is not compatible, hence not tight.

Next suppose that $\text{Red}(w)\cap \text{Cyan}(w)=\varnothing$  but $\text{Red}(w)\cap \text{DarkBlue}(w)\neq \varnothing$.  
Let $T\subseteq L_{\mathfrak{p}}$ be a Dyck rectangle such that  $\text{Red}(w)\cap \text{DarkBlue}(w)\cap T=b_T$. 
Let $u$ be the horizontal edge whose left endpoint is at the southwestern corner of $T$. Let $v$ be the vertical edge whose upper endpoint is at the northeastern corner of $T$.  For any dark blue tile, there is a cyan tile that is to the right of the dark blue tile. Hence $\text{Red}(w)\cap \text{Cyan}(w)\neq\varnothing$, which is a contradiction.  This means that if $\text{Red}(w)\cap \text{Cyan}(w)=\varnothing$  then $\text{Red}(w)\cap \text{DarkBlue}(w)= \varnothing$, and by symmetry, $\text{DarkRed}(w)\cap \text{Cyan}(w)= \varnothing$.

Therefore, if $\omega$ is tight, then $\text{Red}(w)\cap \text{Cyan}(w)=\text{Red}(w)\cap \text{DarkBlue}(w)=\text{DarkRed}(w)\cap \text{Cyan}(w)=\varnothing$.

The converse is straightforward, because if $\text{Red}(w)\cap \text{Cyan}(w)=\text{Red}(w)\cap \text{DarkBlue}(w)=\text{DarkRed}(w)\cap \text{Cyan}(w)=\varnothing$ then for every pair of $u\in \calP_\east$ and $v\in \calP_\north$ with $\omega(u)\omega(v)>0$, it is easy to find an edge $e$ along the subpath $\overrightarrow{uv}$ so that at least one of \eqref{0407df:comp} holds.  The condition that there is some choice of $\mathfrak{p} \in P$ such that $\omega_{\mathfrak{p}}(S_\north) = p$ and $\omega_{\mathfrak{p}}(S_\east) = q$ is necessary by \Cref{rem: frame of tight}. 
\end{proof}

\begin{rem}
The red and cyan tiles are precisely the tiles that appear in the tiling construction described in \cite{BLM} and \cite{BLM2}.  For any horizontal edge on $\calP(d_1,d_2)$, the total height of the cyan tiles in the column above it is precisely the size of its local shadow, and analogously for the red tiles in the row containing each vertical edge.  The dark red and dark blue tiles (as well as their purple intersection) did not appear in any of the previous tiling constructions.
\end{rem}

\section{\texorpdfstring{Proof of \Cref{thm: main 1 intro}}{Proof of theorem}}\label{section: proof}

Recall the general scattering diagram $\frakD(P_1, P_2)$ at the beginning of \Cref{section: change of lattice} with
\[
    P_i = 1 + \sum_{k\geq 1} p_{i,k}x_i^k \in \Bbbk\llbracket x_i\rrbracket .
\]
By treating $p_{i,k}$ as indeterminates, we regard $\frakD(P_1, P_2)$ as a kind of \emph{universal scattering diagram} since any scattering diagram in the context of \Cref{section: change of lattice} can be specialized from $\frakD(P_1, P_2)$ by taking particular values of $p_{i,k}$ in $\Bbbk$; see \cite[Section 5.6]{BLM2} for an algebraic formalization. Now part (1) of \Cref{thm: main 1 intro} readily follows from \Cref{thm: universal coefficient} via specialization.

\begin{proof}[Proof of \Cref{thm: main 1 intro}(1)]
    The scattering diagram $\frakD^\circ_{(c,b)} = \frakD((1+x_1)^c, (1+x_2)^b)$ is obtained from specializing the universal scattering diagram at
    \[
        p_{1,k} = \binom{c}{k} \quad \text{and} \quad p_{2,k}= \binom{b}{k}.
    \]
    By the formula of $\fclu{d}{e}$ given in \Cref{thm: universal coefficient}, we see that
    \[
        \tvco_k = \tvco(kd,ke) = \sum_{\pt_1\vdash kd}\sum_{\pt_2\vdash ke} \lambda(\pt_1,\pt_2) \prod_{p\in \pt_1} \binom{c}{p}\prod_{q\in \pt_2} \binom{b}{q}.
    \]
    Clearly $\tvco_k$ is a polynomial in $b$ and $c$, and $bc$ divides $\tvco_k$.

    By \Cref{cor: length one partitions}, the leading term $\lambda(\pt_1, \pt_2)$ for $\pt_1 = (kd)$ and $\pt_2 = (ke)$ equals $k = \gcd(kd, ke)$.
\end{proof}

Now that the polynomiality of $\tvco_k$ is proven, the polynomiality of $\tau_k$ (\Cref{thm: main 1 intro}(2)) is obtained through the change-of-lattice transformation in \Cref{prop: change of lattice}.

\begin{proof}[Proof of \Cref{thm: main 1 intro}(2)]
    The claimed expression of $\tau_k$ in terms of $\tvco_k$ is a direct consequence of
    \[
        1 + \sum_{k\geq 1}\tau_kx_1^{kd}x_2^{ke} = \left(1 + \sum_{k\geq 1} \tvco_k x_1^{kd}x_2^{ke} \right)^{g/bc} \quad \text{(by \Cref{prop: change of lattice})}.
    \]
    Since every $\tvco_k$ is divided by $bc$, we have $\tau_k\in \mathbb{Q}[b,c,g]$. The top-degree term of $\tau_k$ in $g$ is given by 
    \[
        \left(\frac{\tvco_1}{bc}\right)^k g^k.
    \]
    As $\tvco_1$ is always non-zero by \Cref{thm: main 1 intro}(1), the degree of $\tau_k$ in $g$ is $k$.
\end{proof}

Now write $\tau_k$ as a polynomial in $g$:
\[
    \tau_k = \sum_{n=1}^k \tau_k(n) g^n, \quad \tau_k(n)\in \mathbb{Q}[b,c].
\]
Before showing \Cref{thm: main 1 intro}(3), we first prove \Cref{thm: main 1 intro}(4):
\[
    \tau_k(n) = \frac{1}{n!}\sum_{k_1+\cdots +k_n = k}\tau_{k_1}(1)\cdots \tau_{k_n}(1).
\]
It is another algebraic consequence directly from the change-of-lattice transformation (\Cref{prop: change of lattice}).

\begin{proof}[Proof of \Cref{thm: main 1 intro}(4)]
    For fixed $d$ and $e$, define the generating function
    \[
        F(x) \coloneqq 1 + \sum_{k\geq 1} \tvco_k x^k.
    \]
    By \Cref{prop: change of lattice}, we have
    \[
        1 + \sum_{k\geq 1} \tau_k x^k = F(x)^{g/bc} = \exp\left(g\frac{\ln F(x)}{bc}\right).
    \]
    Expand both sides as series of $g$:
    \[
        1 + \sum_{n\geq 1} g^n \sum_{k\geq 1} \tau_k(n) x^k = 1 + \sum_{n\geq 1} g^n \frac{1}{n!}\left(\frac{\ln F(x)}{bc}\right)^n.
    \]
    Comparing the $g^n$-coefficients between both sides, we have for any $n\geq 1$,
    \[
        \sum_{k\geq 1}\tau_k(n)x^k = \frac{1}{n!}\left(\sum_{k\geq 1}\tau_k(1)x^k\right)^n.
    \]
    Comparing the $x^k$-coefficients of both sides yields the desired equality.
\end{proof}

\begin{rem}\label{rem: rho from tau}
    Evaluating $g=bc$ recovers $\tvco_k$ from $\tau_k$ by
    \[
        \tvco_k = \sum_{n=1}^k (bc)^n\tau_k(n).
    \]
    In particular, in the coprime case we have $\tvco(d,e) = \frac{bc}{g}\tau(d,e)$.
\end{rem}

We proceed to prove part (3) of \Cref{thm: main 1 intro}. Write $(\ii, \jj) = (kd, ke)$ and thus $\tau_k = \tau(\ii, \jj)$. It is clear from \Cref{thm: main 1 intro}(2) that
\begin{equation}\label{eq: b c degree bound}
    \deg_b\tau(\ii, \jj) \leq ke-1=\jj-1 \quad \text{and} \quad \deg_c \tau(\ii,\jj) \leq kd-1=\ii-1. 
\end{equation}
We show these degree bound are attained using a tight grading interpretation. We look at ``limits'' of tight gradings, in the sense that at least one of $b$ or $c$ is taken to be arbitrarily large.

\begin{lem}\label{lem: b c degree tight gradings}
The function $\tau(\ii,\jj) \in \mathbb{Q}[b,c,g]$ has degree $\jj - 1$ in $b$ and $\ii - 1$ in $c$.  Moreover, the coefficient of $gb^{\jj-1}c^{\ii-1}$ in $\tau(\ii,\jj)$ is positive.  
\end{lem}
\begin{proof}
Fix $\ii$, $\jj$, and $g$.  We can hold these constant by choosing a suitable sequence of pairs $(b,c)$ tending to infinity in both entries.  We view $\tau(\ii,\jj)$ as the number of tight gradings in $\CGbc(d_1,d_2,\ii,\jj)$ for an appropriate choice of $d_1,d_2$, following \Cref{rem: tight grading count}.  (Recall that $\CGbc(d_1,d_2,\ii,\jj)$ is the subset of $\CG(d_1,d_2,bi,cj)$ such that 
$\omega$ takes at most one distinct nonzero value on the horizontal edges and at most one distinct nonzero value on the vertical edges.)
We then fix a frame $L_{\mathfrak{p}}$ in $\calP(d_1,d_2)$, as defined in \Cref{subsec: frames}.  There is a constant number of choices of frame, namely $g \cdot \gcd(i,j) = \gcd(b\ii,c\jj)$, according to \Cref{lem: tight grading bounding labels}. 
Moreover, \Cref{lem: tight grading bounding labels} guarantees that the leftmost horizontal edge in $S_\east$ and the topmost vertical edge in $S_\north$ are fixed by the choice of frame. 
Since the $\jj-1$ additional horizontal edges of $S_\east$ must be chosen from the remaining $\ii b-1$ edges in the horizontal shadow, the degree in $b$ is at most $\jj-1$.  Similarly the degree in $c$ is at most $\ii-1$.

Suppose we choose the $\ii-1$ other vertical edges in $S_\north$ arbitrarily among the $\frac{c}{2}$ topmost vertical edges within the rectangular support.  There are $\binom{c/2}{\ii-1}$ such choices.  The red and dark red tiles in the second-to-topmost row have total length $(\ii-1)b$ and is horizontal distance at most $\frac{\ii b}{2\jj} \pm o(b)$ from the topmost vertical edge.  Thus, there are $b - \frac{\ii b}{2\jj} \pm o(b)$ columns that do not have any red or dark red tiles except in the topmost row.  Hence the remaining $\jj-1$ edges of $S_\east$ may be picked arbitrarily among the $\left(1-\frac{\ii}{2\jj}\right)b$ leftmost unchosen horizontal edges in the support.   Note that there are no red or dark red tiles in these leftmost $\left(1-\frac{\ii}{2\jj}\right)b$ columns except for in the topmost row, so there are no overlaps between tiles and hence no purple tiles arise.  There are $\binom{\left(1-\frac{\ii}{2\jj}\right)b}{\jj-1}$ choices for $S_{
\east}$ within these columns, as well as $g \cdot \gcd(i,j) \geq g$ choices for the frame. Thus there are at least a positive constant times $gb^{j-1}c^{i-1}$ tight gradings obtained via this construction.
\end{proof}

\begin{rem}
The proof of \Cref{lem: b c degree tight gradings} yields an explicit lower bound on the coefficient of $gb^{\jj-1}c^{\ii-1}$ in $\tau(\ii,\jj)$.  Instead of choosing $S_\north$ among the top $\frac{c}{2}$ vertical edges, we can introduce a parameter $\lambda \in (0,1)$ and choose $S_\north$ from among the top $\lambda c$ vertical edges.  Then optimizing over all choices of $\lambda$ gives a lower bound of 
$$\frac{(\jj-1)^{\jj-1}\left(\jj(\ii-1)\right)^{\ii-1}}{\ii^{\ii-1}(\ii+\jj-2)^{\ii+\jj-2}} \cdot \frac{\gcd(\ii,\jj)}{(\ii-1)!(\jj-1)!}$$
on the coefficient of $gb^{\jj-1}c^{\ii-1}$.
 
\end{rem}

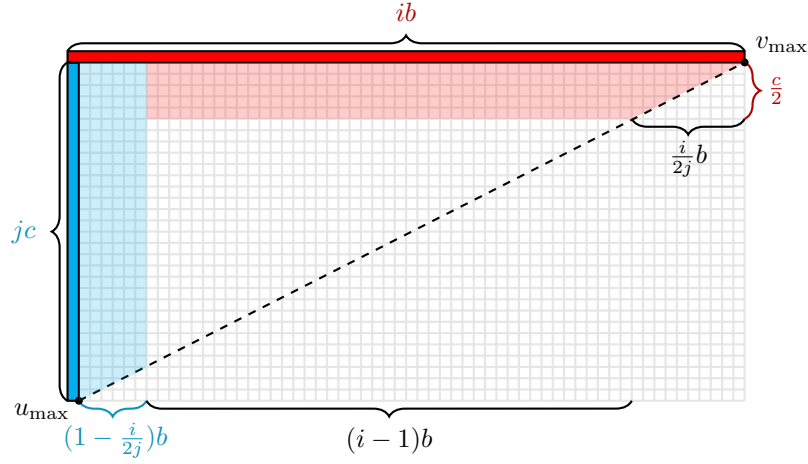
\begin{figure}[ht]
\begin{tikzpicture}[scale=1.5, thick]
  \def\rectWidth{0.1}
  \def\cjHeight{3}
  \def\biLength{6}

  \coordinate (A) at (0,0); 
  \coordinate (B) at ($(A)+(0,\cjHeight)$); 
  \coordinate (C) at ($(B)+(\biLength,0)$); 
  \coordinate (D) at ($(A)+(\biLength,0)$); 

  \draw[step=0.1,color=gray!20!] (A) grid (C);
  \fill[cyan] (A) rectangle ++(\rectWidth,\cjHeight);
  \draw (A) rectangle ++(\rectWidth,\cjHeight);

  \fill[red] (B) rectangle ++(\biLength,\rectWidth);
  \draw (B) rectangle ++(\biLength,\rectWidth);

  \draw[dashed] (A)+(\rectWidth,0) -- (C);

  \draw[decorate,decoration={brace,amplitude=6pt}]
    ($(A)$) -- ($(B)$)
    node[midway,xshift=-8pt,left,cyan!70!black] {\small $\jj c$};

  \draw[decorate,decoration={brace,amplitude=6pt}]
    ($(B)+(0,\rectWidth)$) -- ($(C)+(0,\rectWidth)$)
    node[midway,yshift=15pt,red!70!black] {\small $\ii b$};

     \draw[decorate,decoration={brace,mirror,amplitude=6pt},cyan!70!black]
    ($(A)+(\rectWidth,0)$) -- ($(.7,0)$)
    node[midway,yshift=-15pt,xshift=1pt,cyan!70!black] {\small $(1-\frac{\ii}{2\jj})b$};

    \draw[decorate,decoration={brace,amplitude=6pt},red!70!black]
    ($(C)$) -- ($(C)-(0,.5)$)
    node[midway,xshift=12pt,red!70!black] {\small $\frac{c}{2}$};

    \draw[decorate,decoration={brace,mirror,amplitude=6pt}]
    ($(C)-(1,.5)$) -- ($(C)-(0,.5)$)
    node[midway,yshift=-15pt] {\small $\frac{\ii}{2\jj}b$};

    \draw[decorate,decoration={brace,mirror,amplitude=6pt}]
    ($(.7,0)$) --($(C)-(1,\cjHeight)$) 
    node[midway,yshift=-15pt] {\small $(\ii-1)b$};

  \fill (A)+(\rectWidth,0) circle (1pt) node[yshift=-4pt,xshift=-14pt] {\small $u_{\max}$};
  \fill (C) circle (1pt) node[above right] {\small $v_{\max}$};

  \fill[red,opacity=0.2] (C)--($(C)-(1,.5)$)--(.7,2.5)--(.7,3)--(C);

  \fill[cyan,opacity=0.2] (\rectWidth,0)--(A)+(.7,.29)--(.7,3)--(\rectWidth,3)--(\rectWidth,0);

\end{tikzpicture}
\label{fig: limit path g factor}
\caption{An illustration of the proof of \autoref{lem: b c degree tight gradings}.  Following the construction, the remaining red and dark red tiles are contained in the red shaded region, while the remaining cyan and dark blue tiles are contained in the cyan shaded region.}
\end{figure}

\begin{proof}[Proof of \Cref{thm: main 1 intro}(3)]
    Combining \Cref{lem: b c degree tight gradings} and \eqref{eq: b c degree bound}, we can confirm that 
    $$\deg_b\tau(kd,ke;1)= ke-1 \text{ and }\deg_c\tau(kd,ke;1)= kd-1\,.$$
    Moreover, \Cref{lem: b c degree tight gradings} asserts that the coefficient of $b^{ke-1}c^{kd-1}$ in $\tau(kd,ke;1)$ is positive.  Then focusing on the top-degree term of each factor $\tau(qd,qe;1)$ in part (4) of \Cref{thm: main 1 intro}, we can see that the top-degree term of $\tau(kd,ke;n)$ is a positive constant times $b^{ke-n}c^{kd-n}$. 
\end{proof}

We end this section by the following example. The implications of \Cref{C16,C17} are discussed at the end of the example.

\begin{exmp}\label{ex: d=1 e=j}
    Let $(d, e) = (1, \jj)$. By \Cref{rem: rho from tau}, we have
    \[
        \tvco(2, 2\jj) = (bc)^2 \tau(2, 2\jj; 2) + bc\cdot \tau(2, 2\jj;1),
    \]
    where $\tau(2, 2\jj; 2) = \frac{1}{2}\tau(1,\jj;1)^2$ by \Cref{thm: main 1 intro}(4) and $\tau(1, \jj; 1) = \tvco(1, \jj)/bc$ (the coprime case of \Cref{rem: rho from tau}). Thus we can express
    \begin{equation}\label{eq: tau 2 2j 1}
        \tau(2, 2\jj;1) = \frac{1}{bc}\tvco(2,2\jj) - \frac{1}{2bc} \tvco(1,\jj)^2.
    \end{equation}
    Next we determine $\tvco(1,\jj)$ and $\tvco(2, 2\jj)$ explicitly:
    \begin{align}
        \tvco(1,\jj) &= c\binom{b}{\jj}, \label{eq: rho 1 j}\\
        \tvco(2,2\jj) &= \binom{b}{\jj}^2\binom{c}{2} + \sum_{n=1}^\jj 2\binom{b}{\jj+n}\binom{b}{\jj-n}\binom{c}{2} + \sum_{n=1}^{\jj} (2n-1)\binom{b}{\jj+n}\binom{b}{\jj-n}c^2. \label{eq: rho 2 2j}
    \end{align}

    Recall from \Cref{thm: main 1 intro}(1) that
    \[
        \tvco(kd, ke) = \sum_{\pt_1\vdash kd}\sum_{\pt_2\vdash ke} \lambda(\pt_1, \pt_2) \prod_{p\in \pt_1}\binom{c}{p}\prod_{q\in \pt_2}\binom{b}{q}.
    \]
    For $\tvco(1,\jj)$, the result follows from that there is a unique tight grading associated with the partitions $\pt_1=(1)$ and $\pt_2=(\jj)$, giving $\lambda((1), (\jj))=1$ (\Cref{cor: length one partitions}) and that there is none for any other pairs of partitions.

    For $\tvco(2, 2\jj)$, we list all non-zero $\lambda(\pt_1, \pt_2)$ with $\pt_1\vdash 2$ and $\pt_2\vdash 2\jj$:
    \begin{enumerate}
        \item $\lambda((2), (\jj,\jj)) = 1$;
        \item $\lambda((2), (\jj+n,\jj-n)) = 2$ for $n = 1, \dots, \jj$;
        \item $\lambda((1,1), (\jj+n,\jj-n)) = 2n-1$ for $n=1, \dots, \jj$.
    \end{enumerate}
    These counts are obtained from finding all corresponding tight gradings. See \Cref{fig: j=5 p1=1+1} and \Cref{fig: j=5 p1=2} for the case where $j=5$. Then the formula \eqref{eq: rho 2 2j} follows.

    \begin{figure}[ht]
    \centering
    \includegraphics[width=\linewidth]{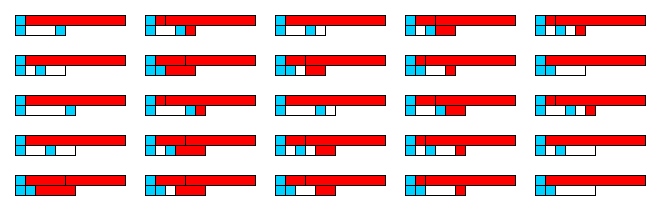}
    \caption{Tight gradings in the case $j=5$ with $\mathbf{P}_1 = (1,1)$.}
    \label{fig: j=5 p1=1+1}
    \end{figure}
    
    \begin{figure}[ht]
        \centering
        \includegraphics[width=\linewidth]{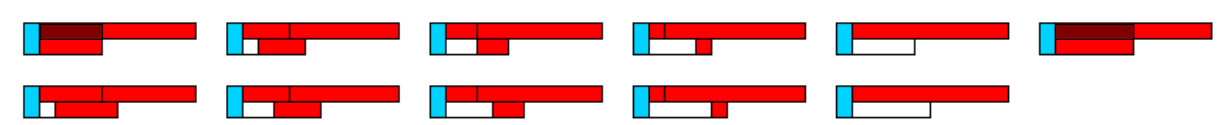}
        \caption{Tight gradings in the case $j=5$ with $\mathbf{P}_1 = (2)$.}
        \label{fig: j=5 p1=2}
    \end{figure}

    By symmetry, we also obtain formulas for $\tvco(\ii, 1)$ and $\tvco(2\ii, 2)$ from \eqref{eq: rho 1 j} and \eqref{eq: rho 2 2j} by replacing $\jj$ with $\ii$ and swapping $b$ and $c$.

    Finally we note by \Cref{thm: main 1 intro}(4),
    \[
        \tau(k, \jj k; k-1) = \frac{\tau(1,\jj;1)^{k-2}\tau(2,2\jj;1)}{(k-2)!}.
    \]
    However by the formulas \eqref{eq: rho 1 j}, \eqref{eq: rho 2 2j} and \eqref{eq: tau 2 2j 1}, clearly $\tau(1,\jj;1)$ divides $\tau(2,2\jj;1)$. This proves \Cref{C16} as well as \Cref{C17} by symmetry.
\end{exmp}

\begin{rem}
    Notice that there are no dark blue or dark red tiles in \Cref{fig: j=5 p1=1+1}.
    In \Cref{fig: j=5 p1=1+1}, the tight gradings on the first two rows share the same frame, while those on the last three rows share a different frame.
\end{rem}

\section{Weyl group symmetry}\label{section: weyl symmetry}
In this section, using tight gradings, we give a bijective proof of the Weyl group symmetry for cluster scattering diagrams, namely
\begin{equation}\label{eq: weyl symmetry}
    \tau^{b,c}(\ii,\jj) = \tau^{b,c}(\ii,b\ii-\jj) \text{ and } \tau^{b,c}(\ii,\jj) = \tau^{b,c}(c\jj-\ii,\jj)\,.
\end{equation}
This resolves \Cref{problem_of_ERS} posed by Elgin, Reading, and Stella \cite{ERS}.

\subsection{Mutation and Retraction}
One of the main tools we use is the Lee--Li--Zelevinsky mutation map \cite{LLZ}.  Combining this mutation map with a new ``retraction'' map yields a bijection between collections of tight gradings that exhibit the Weyl group symmetries.

We start by recalling some of the notions from \cite{LLZ}.  

The \emph{remote shadow of $S_\north$}, denoted by $\rsh(S_\north)$, is the subset of $\sh(S_\north)$ obtained by removing all horizontal edges of height $\fpj-1$ whenever $v_\fpj \in S_\north$. We say that $\edge \in \calP$ is \emph{directly shadowed} by $\edge' \in \calP$ if $\edge \in \sh(\edge')$ and $\sh(\edge')$ is minimal with respect to this property.  Note that this is well defined because any pair of shadows are either disjoint or nested.  For $0 \leq h < d_2$ and $0 < \fpj \leq d_2$, let $\rsh(S_\north)_{h;\fpj}$ denote the set of edges in $\rsh(S_\north)$ of height $h$ that are directly shadowed by $v_\fpj$.  For $0 < \fpi \leq |\rsh(S_\north)_{h;\fpj}|$ let $\rsh(S_\north)_{h;\fpj;\fpi}$ denote the $\fpi^{\text{th}}$ element of $\rsh(S_\north)_{h;\fpj}$, ordered left-to-right.

Since a grading $\omega \in \CGbc(d_1,d_2,i,j)$ is completely determined by the edge sets $S_\east$ and $S_\north$ where $\omega$ is nonzero, from now on we identify such $\omega$ with the pair $(S_\east,S_\north)$.  This is precisely the setting of \emph{compatible pairs} in \cite{LLZ}; see also the discussion below \Cref{def: compatible grading}. 

\begin{defn}[{\cite[Lemma 3.5]{LLZ}}]\label{defn: compatible pair mutation}
    The \emph{mutation map}
    \[
        \theta: \CGbc(d_1,d_2,i,j) \to \CGbc(bd_2 - d_1, d_2)
    \]
    is defined for compatible gradings follows.
    Fix $(S_\east,S_\north) \in \CGbc(d_1,d_2,i,j)$.  Then we have $\theta(S_\east,S_\north) = (S_\east',S_\north')$, where 
    $$S_\north' = \{v_\fpj' \, : \, v_{d_2+1-\fpj}\in \mathcal{P}_2\setminus S_\north \}, \text{ and }$$
    $$S_\east' = \{\rsh(S_\north')_{d_2 - \fpj; d_2 - h;\fpi} : \rsh(S_\north)_{h;\fpj;\fpi} \in S_\east\}.$$
\end{defn}

The fact that the image of $\theta$ consists of compatible gradings follows from \cite[Lemmas 3.5 and 3.12]{LLZ}.  Note that $\theta$ forgets the information of the edges in $\calP_1 \cut \sh(S_\north)$.

\begin{rem}\label{rem: mutation map comparison}
Our definition of the Lee--Li--Zelevinsky map $\theta$ differs from their original formulation in \cite{LLZ}.  We apply $\theta$ to compatible gradings, rather than sets of vertical edges.  It is straightforward to show that the two definitions are equivalent up to specifying a choice of horizontal edges.  Note that we use the notation $d_1$, $d_2$ in place of their parameters $a_1$, $a_2$.
\end{rem}

In the definition of tight gradings, we have the condition that $b\ii d_2 - c\jj d_1 = \pm \gcd(b\ii, c\jj)$.  The choice of sign determines the shadow structure of the tight grading, namely that $S_\east \subseteq \sh(S_\north)$ when $b\ii d_2 - c\jj d_1  = \gcd(b\ii, c\jj)$ and $S_\north \subseteq \sh(S_\east)$ otherwise.  We thus consider the following sets of \emph{shadowed gradings} 

\begin{align*}
    \CGbc^\top(d_1, d_2, \ii, \jj) &\coloneqq \{\omega \in \CGbc(d_1, d_2, \ii, \jj) : S_\east \subseteq \sh(S_\north)\}\,, \text{ and }\\
    \CGbc^\vdash(d_1, d_2, \ii, \jj) &\coloneqq \{\omega \in \CGbc(d_1, d_2, \ii, \jj) : S_\north \subseteq \sh(S_\east)\}\,.
\end{align*}

\begin{prop}\label{prop: compatibility mutation}
The mutation map $\theta$ from \autoref{defn: compatible pair mutation} is injective when restricted to the set $\CGbc^\top(d_1,d_2,\ii,\jj)$. 
\end{prop}
\begin{proof}
    When $(S_\east,S_\north) \in \CGbc^\top(d_1,d_2,\ii, \jj)$, then we have $S_\east \subseteq \rsh(S_\north)$.  Thus all the information about $S_\east$ is preserved under mutation.
\end{proof}
    
In addition to mutation, we require a retraction map for compatible gradings that forgets some of the vertical edges, namely those outside the shadow of the horizontal edges.  This allows us to obtain shadowed gradings from compatible gradings, which will then enable us to get a tight grading in $\CGbc^\vdash$ from one in $\CGbc^\top$.   

\begin{defn}
The retraction map 
$$\varrho: \CGbc(d_1,d_2,\ii, \jj) \to \coprod_{0 \leq k \leq \ii} \CGbc^{\vdash}(d_1,d_2,k,\jj)$$ maps a compatible grading $(S_\east,S_\north)$ to $(S_\east,S_\north \cap \sh(S_\east))$.
\end{defn}

\subsection{Preservation on Tight Gradings}

We now show that the Weyl group symmetry $\tau^{b,c}(\ii,\jj) = \tau^{b,c}(c\jj-\ii,\jj)$ is exhibited by the map  $\varrho \circ \theta$, i.e., the composition of mutation and retraction, on the domain $\CGbc^\top(d_1,d_2,\ii, \jj)$.  The other symmetry can be obtained by an analogous map, swapping the roles of vertical and horizontal edges.

\begin{exmp}
Let $b = 3$, $c = 2$, $(d_1,d_2,\ii,\jj)=(7,5,2,2)$. Then $\varrho\circ \theta$ acts on a particular tight grading as shown in \Cref{fig: mutation example}.

\begin{figure}[ht]
\begin{tikzpicture}[scale=.4]
        \draw[step=1,color=gray] (0,0) grid (7,5);
        \draw (0,0)--(7,5);
        \draw[line width=1.5,color=black] (0,0)--(2,0)--(2,1)--(3,1)--(3,2)--(5,2)--(5,3)--(6,3)--(6,4)--(7,4)--(7,5);
        \fill[cyan] (1.2,0.2)--(1.2,1.8)--(1.8,1.8)--(1.8,0.2)--(1.2,0.2);
        \fill[cyan] (2+0.2,1.2)--(2+0.2,2.8)--(2+0.8,2.8)--(2+0.8,1.2)--(2+0.2,1.2);
        \fill[cyan] (1+0.2,2+0.2)--(1+0.2,3+0.8)--(1+0.8,3+0.8)--(1+0.8,2+0.2)--(1+0.2,2+0.2);
        \fill[red] (1.2,4+0.2)--(1.2,4+0.8)--(3.8,4+0.8)--(3.8,4+0.2)--(1.2,4+0.2);
        \fill[red] (3.2,3+0.2)--(3.2,3+0.8)--(5.8,3+0.8)--(5.8,3+0.2)--(3.2,3+0.2);
        \fill[red] (3+1.2,4+0.2)--(3+1.2,4+0.8)--(3+3.8,4+0.8)--(3+3.8,4+0.2)--(3+1.2,4+0.2);
        \draw[->] (8,2.5)--(10.5,2.5)
        node[midway,yshift=7pt] {$\theta$};
    \begin{scope}[shift={(11.5,0)}]
        \draw[step=1,color=gray] (0,0) grid (8,5);
        \draw (0,0)--(8,5);
        \draw[line width=1.5,color=black] (0,0)--(2,0)--(2,1)--(4,1)--(4,2)--(5,2)--(5,3)--(7,3)--(7,4)--(8,4)--(8,5);
                \fill[red] (1.2+1,4+0.2)--(1.2+1,4+0.8)--(3.8+1,4+0.8)--(3.8+1,4+0.2)--(1.2+1,4+0.2);
        \fill[red] (3.2+1,3+0.2)--(3.2+1,3+0.8)--(5.8+1,3+0.8)--(5.8+1,3+0.2)--(3.2+1,3+0.2);
        \fill[red] (3.2+1-3,3+0.2)--(3.2+1-3,3+0.8)--(5.8+1-3,3+0.8)--(5.8+1-3,3+0.2)--(3.2+1-3,3+0.2);
        \fill[red] (3+1.2+1,4+0.2)--(3+1.2+1,4+0.8)--(3+3.8+1,4+0.8)--(3+3.8+1,4+0.2)--(3+1.2+1,4+0.2);
         \fill[red] (1.2+1,2+0.2)--(1.2+1,2+0.8)--(3.8+1,2+0.8)--(3.8+1,2+0.2)--(1.2+1,2+0.2);
         \fill[red] (1.2+1-3,4+0.2)--(1.2+1-3,4+0.8)--(3.8+1-3,4+0.8)--(3.8+1-3,4+0.2)--(1.2+1-3,4+0.2);
        \fill[cyan] (1.2,0.2)--(1.2,1.8)--(1.8,1.8)--(1.8,0.2)--(1.2,0.2);
        \fill[cyan] (0.2,0.2)--(0.2,1.8)--(0.8,1.8)--(0.8,0.2)--(0.2,0.2);
        \fill[cyan] (0.2,2+0.2)--(0.2,2+1.8)--(0.8,2+1.8)--(0.8,2+0.2)--(0.2,2+0.2);
        \draw[->] (9,2.5)--(11.5,2.5)
        node[midway,yshift=7pt] {$\varrho$};
    \end{scope}
    \begin{scope}[shift={(24,0)}]
        \draw[step=1,color=gray] (0,0) grid (8,5);
        \draw (0,0)--(8,5);
        \draw[line width=1.5,color=black] (0,0)--(2,0)--(2,1)--(4,1)--(4,2)--(5,2)--(5,3)--(7,3)--(7,4)--(8,4)--(8,5);
        \fill[red] (3.2+1-3,3+0.2)--(3.2+1-3,3+0.8)--(5.8+1-3,3+0.8)--(5.8+1-3,3+0.2)--(3.2+1-3,3+0.2);
        \fill[red] (3.2+1,3+0.2)--(3.2+1,3+0.8)--(5.8+1,3+0.8)--(5.8+1,3+0.2)--(3.2+1,3+0.2);
         \fill[red] (1.2+1,2+0.2)--(1.2+1,2+0.8)--(3.8+1,2+0.8)--(3.8+1,2+0.2)--(1.2+1,2+0.2);
        \fill[cyan] (1.2,0.2)--(1.2,1.8)--(1.8,1.8)--(1.8,0.2)--(1.2,0.2);
        \fill[cyan] (0.2,0.2)--(0.2,1.8)--(0.8,1.8)--(0.8,0.2)--(0.2,0.2);
        \fill[cyan] (0.2,2+0.2)--(0.2,2+1.8)--(0.8,2+1.8)--(0.8,2+0.2)--(0.2,2+0.2);
    \end{scope}
    \end{tikzpicture} 
    \caption{The tight grading on the left is mapped to the compatible grading in the center via the mutation map $\theta$, which is in turn mapped to the tight grading on the right by the retraction map $\varrho$.} 
    \label{fig: mutation example}

\end{figure}
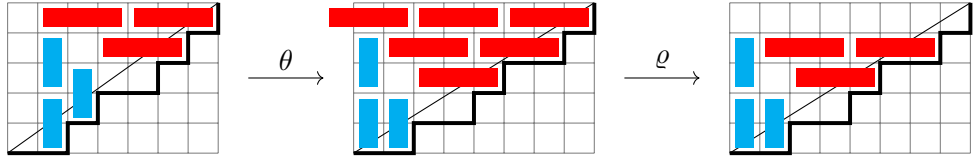
\end{exmp}

Let $\widetilde \omega$ be the grading obtained by applying the mutation map $\theta$ post-composed by the retraction $\varrho$ to $\omega \in \CGbc^\top(d_1,d_2,\ii, \jj)$.  Let $\widetilde S_\east$ and $\widetilde S_\north$ be the horizontal and vertical edges of nonzero weight in $\widetilde \omega$, respectively.  

\begin{lem}\label{lem: mutation retraction compatible}
We have $\widetilde \omega \in \CGbc^\vdash(b d_2 - d_1, d_2, c\jj - \ii, \jj)$.
\end{lem}
\begin{proof}
By the definition of the retraction map, we can readily see that $\widetilde S_\north \subseteq \sh(\widetilde S_\east)$.  Moreover, since retraction only shrinks the set of vertical edges with nonzero grading value, then the compatibility of $\widetilde \omega$ follows directly from that of $\theta(\omega)$, shown in \Cref{prop: compatibility mutation}.   Thus, it only remains to check the sizes of $\widetilde S_\east$ and $\widetilde S_\north$.

By construction, we have $\big|\widetilde S_\east\big| = \big|S_\east\big| = j$.  Following from \cite[Lemma 3.5]{LLZ}, the edges in $\sh(\widetilde S_\east) \supseteq \widetilde S_\north$ correspond to $\{v_{\max}\} \cup \sh(S_\east) \cut \{v_{\min}\}$.  Since $\omega(v_{min}) = 0$ by \Cref{lem: tight grading bounding labels}, the number of edges in $\widetilde S_\north$ is precisely $\big|\sh(\widetilde S_\east)\big| - \ii = c \jj - \ii$.
\end{proof}

\begin{proof}[Proof of \Cref{thm: weyl symmetry}]
It follows from \Cref{lem: mutation retraction compatible} and the fact that 
$$b (c \jj - \ii )d_2 - c \jj (b d_2 - d_1) =  c \jj d_1 - b \ii d_2  = -\gcd(bi,cj)$$
that $\widetilde \omega$ is also a tight grading.  To see that this map is bijective, we first note that injectivity of $\theta$ follows from \Cref{prop: compatibility mutation}, hence the injectivity of $\rho \circ \theta$ then follows \Cref{lem: tight grading bounding labels}.  By applying the analogous map to the tight gradings in $\CGbc^\top(b d_2 - d_1, d_2, c\jj - \ii, \jj)$, we obtain bijectivity.
\end{proof}

\begin{rem}\label{rem: catalan bijection}
The central wall $\mathbb{R}_{\leq 0}(3, 2)$ in $\frakD_{(3,2)}$ is preserved under the Weyl symmetry.  Hence, this shows that the composition of mutation and retraction gives a bijection between the sets of tight gradings in $\CG_{3,2}^{\top}(3i+2,2i+1,i,i)$ and in $\CG_{3,2}^{\vdash}(3i+1,2i+1,i,i)$, each corresponding to coefficients of the central wall-function in $\frakD_{(3,2)}$.
\end{rem}

\section{Counting tight gradings}

\label{sec:countingTG}

Recall from the statement of Theorem  \ref{thm: universal coefficient}(2) that for partitions $\pt_1 \vdash \pp$ and $\pt_2 \vdash \qq$, $\lambda(\pt_1, \pt_2)$ equals the number of tight gradings $\omega$ on $\mathcal P(d_1,d_2)$ with 
$\pt_1 = (\omega(v)\mid v\in S_\north(\omega))$ and $\pt_2 = (\omega(u)\mid u\in S_\east(\omega)),$
as long as $d_1 \geq \pp$ and $d_2 \geq \qq$ are chosen so that $|\pp d_2 - \qq d_1| = \gcd(\pp,\qq)$.
Nonetheless, determining $\lambda(\pt_1, \pt_2)$ by counting tight gradings in general seems to be a very hard combinatorial problem. We first give a closed formula for $\lambda(\pt_1, \pt_2)$ in the case where one of the partitions is of length one, which generalizes \Cref{cor: length one partitions}. Then we derive a total count of tight gradings with $|\pt_1| = |\pt_2| = k$. Both proofs make use, in a sense, of wall-crossing formulas, namely, the additional information that the counts are encoded in consistent scattering diagrams.  We conclude this section with a combinatorial argument in a special case.

For a partition $\pt$, write $\mu(\pt)$ for its tuple of multiplicities and denote
\[
    \binom{z}{\mu(\pt)} = \dfrac{z(z-1)\cdots (z- \ell(\pt) + 1)}{\prod\limits_{m\in \mu(\pt)} m!}.
\]

\begin{prop}\label{prop: formula one length one part}
    For any $\pt \vdash \jj$, we have
    \begin{equation}\label{eq: formula one length one part}
        \lambda((\ii), \pt) = \frac{\gcd(\ii, \jj)}{i}\binom{\ii}{\mu(\pt)}.
    \end{equation}
\end{prop}

\begin{proof}
    Consider the scattering diagrams $\frakD(1+x^i, P_2)$ and $\frakD(1+x, P_2^i)$, where $P_2 = 1 + \sum_{k\geq 1}p_{2,k}y^k$. Note that we have the expansion
    \begin{equation*}
        P_2^i = 1 + \sum_{k\geq 1}\sum_{\mathbf{Q}\vdash k}\binom{i}{\mu(\mathbf{Q})}p_{2}^{\mathbf{Q}}y^k,
    \end{equation*}
    where $p_2^{\mathbf{Q}}$ denotes the monomial $\prod_{k\in \mathbf{Q}}p_{2,k}$.
    
    The two scattering diagrams are related by a change-of-lattice transformation (slightly more general than \Cref{prop: change of lattice} but can be proven similarly): the wall-function $\frakf(x^iy^j)$ on $\mathbb{R}_{\leq 0}(i,j)$ of $\frakD(1+x^i, P_2)$ and the wall-function $\frakf'(xy^j)$ on $\mathbb{R}_{\leq 0}(1, j)$ of $\frakD(1+x, P_2^i)$ are related by
    \[
        \frakf(xy^j) = \frakf'(xy^j)^{\frac{\gcd(i, j)}{i}}.
    \]
    Comparing the $xy^j$-coefficients on both sides yields
    \[
        \sum_{\pt\vdash \jj}\lambda((\ii), \pt)p_2^{\pt} = \frac{\gcd(\ii,\jj)}{\ii}\sum_{\mathbf{Q}\vdash \jj}\lambda((1), \mathbf{Q})\prod_{p\in \mathbf{Q}}\sum_{\pt\vdash p}\binom{\ii}{\mu(\pt)}p_2^{\pt}.
    \]
    However, for any partition $\mathbf{Q}$ with $\ell(\mathbf{Q})>1$, $\lambda((1), \mathbf{Q}) = 0$ by its tight grading interpretation. As a special case of \Cref{cor: length one partitions}, we have $\lambda((1), (\jj)) = 1$. Then the above equality reduces to
    \[
        \sum_{\pt\vdash \jj}\lambda((\ii), \pt)p_2^{\pt} = \frac{\gcd(\ii, \jj)}{\ii}\sum_{\pt\vdash \jj}\binom{\ii}{\mu(\pt)}p_2^{\pt}.
    \]
    The desired formula follows from comparing the $p_2^\pt$-coefficients of each $\pt$ on both sides.
\end{proof}

For example, if $\ell(\pt)>\ii$, then $\lambda((\ii), \pt)$ vanishes; if $\ell(\pt) = 1$, then $\lambda((\ii), (\jj)) = \gcd(\ii, \jj)$. It is yet unclear how to show the formula \eqref{eq: formula one length one part} purely by tight grading combinatorics. One may allude to the analysis in the proof of \Cref{cor: length one partitions}, but notice that in general $\gcd(i,j)$ does not necessarily divide $\lambda((i), \pt)$.  See the end of this section for an example of tight grading combinatorics when $\ell(\pt)=\ii$.

Further, we can derive one total count of tight gradings using the quiver methods of Section \ref{section: wall function quiver moduli}. We consider tight gradings of total weights $(k, k)$ on the maximal Dyck path $\mathcal{P}(k+1,k)$ for $k\geq 1$, which proceeds from $(0,0)$ to $(k+1,k)$ in steps $$E(EN)^k=u_1u_2v_1u_3v_2\ldots u_{k+1}v_k,$$ where $E$ (resp.~$N$) denotes an east (resp.~north) step. 
Our aim is to prove that they are counted by the series \cite[A002293]{OEIS}:

\begin{prop}\label{prop: total num tight grading central}
The number of tight gradings of total weights $(k, k)$ on $\mathcal{P}(k+1,k)$ equals $$ \sum_{\pt_1\vdash k}\sum_{\pt_2\vdash k} \lambda(\pt_1, \pt_2) = \frac{1}{4k+1}\binom{4k+1}{k}=\frac{1}{3k+1}\binom{4k}{k}.$$
\end{prop}

\begin{proof}Let $\ell_1$, $\ell_2$ be arbitrary positive integers. We consider the tropical vertex as in \cite[Section 3]{BLM}, with variables $p_{i,j}$ specialized to
$$p_{i,j}=\binom{\ell_i}{j},\; i=1,2,\, j\geq 1.$$
Then \cite[Theorem 3.5]{BLM} (see also \Cref{thm: universal coefficient}) reads as follows:
$$f_{(1,1)}=1+\sum_{k\geq 1}\sum_\omega{\rm wt}(\omega)(xy)^k,$$
where the inner sum ranges over all tight gradings $\omega$ on $\mathcal{P}(k+1,k)$ as above, and the weight is defined as
$${\rm wt}(\omega)=\prod_{i=1}^{k+1}\binom{\ell_1}{\omega(u_i)}\prod_{j=1}^k\binom{\ell_2}{\omega(v_j)}.$$
On the other hand, we consider again the refined GW/Kronecker correspondence \cite{RW}. Namely, vanishing results for Euler characteristic of quiver moduli yield by \cite[Section 11]{RW}:
$$f_{(1,1)}=H(xy)^{\ell_1\ell_2},$$
where the series $H(z)$ is determined by the functional equation
$$H(z)=(1-zH(z)^{\ell_1\ell_2-\ell_1-\ell_2})^{-1}.$$
Applying the Lagrange inversion formula (similarly to \cite[Corollary 11.2]{RW}) we find
$$f_{(1,1)}=1+\sum_{k\geq 1}\frac{\ell_1\ell_2}{(\ell_1-1)(\ell_2-1)k+\ell_1\ell_2}\binom{(\ell_1-1)(\ell_2-1)k+\ell_1\ell_2}{k}(xy)^k.$$
Comparing coefficients, we find
$$\sum_{\omega}\prod_{i=1}^{k+1}\binom{\ell_1}{\omega(u_i)}\prod_{j=1}^k\binom{\ell_2}{\omega(v_j)}=\frac{\ell_1\ell_2}{(\ell_1-1)(\ell_2-1)k+\ell_1\ell_2}\binom{{(\ell_1-1)(\ell_2-1)k+\ell_1\ell_2}}{k}.$$
We note that both sides behave polynomially in $\ell_1$, $\ell_2$. Since the identity holds for all nonnegative $\ell_1,\ell_2$, it holds as an identity of (degree $k$) polynomials in $\ell_1,\ell_2$, considered as variables. We can therefore specialize $\ell_1,\ell_2$ to $-1$. Using
$$\binom{-1}{k}=(-1)^k,$$we find
$$\sum_{\omega}(-1)^k(-1)^k=\frac{1}{4k+1}\binom{4k+1}{k}.$$
But the left-hand side is just the number of all tight gradings on $\mathcal{P}(k+1,k)$, finishing the proof.\end{proof}

Finally, to give a better sense concerning the combinatorics of tight gradings, we close this section with a second proof of Proposition \ref{prop: formula one length one part} restricted to the special case that $i=j$.

\begin{proof} [Combinatorial Proof of Proposition \ref{prop: formula one length one part} when $i=j$] 

When
$\pt \vdash \ii$, we have $\lambda((\ii), \pt) = \binom{\ii}{\mu(\pt)}$, which can be rewritten as 
$$\binom{\ii}{\mu(\pt)} = \frac{\ii(\ii-1)\cdots (\ii-\ell(\pt)+1)}{\prod_{m \in \mu(\pt)} m!} = \frac{\ii!}{\prod_{m\in \mu(\overline{\pt})} m!}$$ where $\overline{\pt}$ is the partition $\pt$ extended to length $\ii$ by appending $(\ii - \ell)$ zeroes.  Consequently, $\lambda((\ii), \pt) = \binom{\ii}{\mu(\pt)}$ counts the number of rearrangements of the extended partition $\overline{\pt}$ of length $i$ (including zeroes). 

We demonstrate the following bijection between the set of tight gradings on $\mathcal{P}(\ii+1,\ii)$ with $\pt_1 = (\ii)$ and $\pt_2 \vdash \ii$ and the set of rearrangements of $\overline{\pt_2}$:

1) The Maximal Dyck path $\mathcal{P}(\ii+1,\ii)$ has the form $E(EN)^{\ii}$.  Given the composition $(a_1,a_2,\dots, a_\ii)$ which is a rearrangement of $\overline{\pt_2}$, such that $0 \leq a_j \leq \ii$ and $a_1+a_2+\dots+a_\ii = \ii$, we show there is a unique tight grading with $\omega_2(v_j) = a_j$ where $v_j$ is the $j$th $N$ step of $\mathcal{P}(\ii+1,\ii)$.

2) We claim by use of the Cycle Lemma \cite{cycle-lemma}, there is a unique choice of $0 \leq h \leq \ii-1$ such that we have the following inequalities satisfied by partial sums:
$$\mathrm{For~}1\leq g \leq h,~~\sum_{m=1-h+\ii}^{g-h+\ii} a_m \leq (g-1),~~~\mathrm{and~for~}h < g \leq \ii,~~\left(\sum_{m=1-h+i}^\ii a_m + \sum_{m=1}^{g-h} a_m\right) \leq g.
$$

To see this, we first extend the composition $(a_1,a_2,\dots, a_\ii)$ to one of length $\ii+1$ by assigning $a_0=0$, and then turn $(a_0,a_1,a_2,\dots, a_\ii)$ into a lattice path from $(0,0)$ to $(\ii+1,\ii)$ of the form
$E(N^{a_0})E(N^{a_1})E(N^{a_2})\cdots E(N^{a_i})$, i.e. for $0 \leq g \leq \ii$, alternate between a single $E$ step and a sequence of $a_g$ $N$ steps.  Since $a_0 +a_1+a_2+\dots + a_\ii = \ii$, we reach $(\ii+1,\ii)$ as desired.  The Cycle Lemma states that there is a unique cyclic rotation of the steps of this lattice path that yields a Dyck path.  That rotation yielding the Dyck path equivalently yields a rotated tuple 
$(a_j,a_{j+1},\dots,a_{\ii-1}, a_\ii, a_0,a_1,a_2,\dots, a_{j-1})$ satisfying the system of inequalities $a_j \leq 0$, $a_j + a_{j+1} \leq 1$, $a_j + a_{j+1} + a_{j+2} \leq 2$, and continuing so that each length $g$ partial sum is less or equal to $g-1$. Removing the entry for $a_0=0$, we obtain the above system of inequalities after letting $j= i+1-h$.

3) The points closest to the diagonal of $\mathcal{P}(\ii+1,\ii)$ are $p_0 = (0,0)$, $p_1 =(\ii,\ii-1), p_2 = (\ii-1,\ii-2),\dots, p_{\ii-1} = (2,1)$.  Starting with a composition $(a_1,a_2,\dots, a_\ii)$ as above, we identify the unique choice of $h$ (in the range $0 \leq h \leq \ii-1$) such that the above system of inequalities is satisfied.  Then we work in the frame starting at point $p_h$ but for $1 \leq j \leq \ii$, let $\omega_2(v_j) = a_j$ where $v_j$ is the $j$th $N$ step of $\mathcal{P}(\ii+1,\ii)$, so that $v_1$ is the $N$ step after the unique $EE$ subsequence of $\mathcal{P}(\ii+1,\ii)$.  If $\ell(\pt_2) = \rho$, we get $\rho(\rho+1)/2$ tiles as part of the horizontal footprint attached to these vertical edges.

By construction of this tight grading, starting the frame from $p_h$, this means we have tiles of length $a_{\ii+1-h}$ in the bottom row, length $a_{\ii+1-h}+a_{\ii+2-h}$ in the row $2$nd from the bottom, and continuing so we have tiles of length $a_{\ii+1-h}+a_{\ii+2-h}+\dots +a_{\ii+j-h}$ in the row $j$th from the bottom, and length $a_1+a_2+\dots + a_\ii = \ii$ on the top row.

By the corresponding system of inequalities for this choice of $h$, we have exactly enough room for the vertical tile of length $\ii$ corresponding to partition $\pt_1 = (\ii)$ on the leftmost horizontal edge of this frame.  
\end{proof}

We illustrate this proof with \Cref{exI3}.

\begin{exmp} 
\label{exI3}
Suppose that $\pt_1 = (\ii)=(3)$ and $\pt_2 \vdash 3$.  We work with the maximal Dyck path $\mathcal{P}(4,3)$ of the form $EENENEN$.  If $\pt_2 = (3)$, $\overline{\pt_2} = (3,0,0)$ and the possible rearrangements are $(a_1,a_2,a_3) = (0,0,3), (0,3,0)$, or $(3,0,0)$.  The corresponding tight gradings are illustrated, respectively, in the top row of \Cref{fig:iEq3}.  On the other hand, if $\pt_2 = (2,1)$, $\overline{\pt_2} = (2,1,0)$ and there are six possible rearrangements:
$(a_1,a_2,a_3) = (0,2,1), (2,1,0), (0,1,2), (1,2,0),$ $(1,0,2)$, or $(2,0,1)$.  The corresponding tight gradings are illustrated with respect to this order, in reading order, in the second, third, and fourth row of \Cref{fig:iEq3}.  Finally, if $\pt_2 = \overline{\pt_2} = (1,1,1)$, then $(a_1,a_2,a_3) = (1,1,1)$ is the only possible rearrangement, and the associated tight grading appears in the fifth row of \Cref{fig:iEq3}. 

The arrangement of these tight gradings in \Cref{fig:iEq3} was also chosen so that relative to the use of the Cycle Lemma in Step 2, $h=0$ for the tight gradings (and corresponding $3$-tuples $(a_1,a_2,a_3)$) in the leftmost column, $h=1$ for those in the middle column, and $h=2$ for those in the rightmost column.
Also, by adding in $a_0=0$, we see that each row is an orbit under rotation of $(a_0,a_1,a_2,a_3)$.
In particular $(0,0,2,1)$ and $(0,2,1,0)$ comprise the 2nd line, $(0,0,1,2)$ and $(0,1,2,0)$ comprise the 3rd line, and 
$(0,1,0,2)$ and $(0,2,0,1)$ comprise the 4th line. 
\end{exmp}

\begin{figure}
\begin{center}
\includegraphics[width=0.35\textwidth]{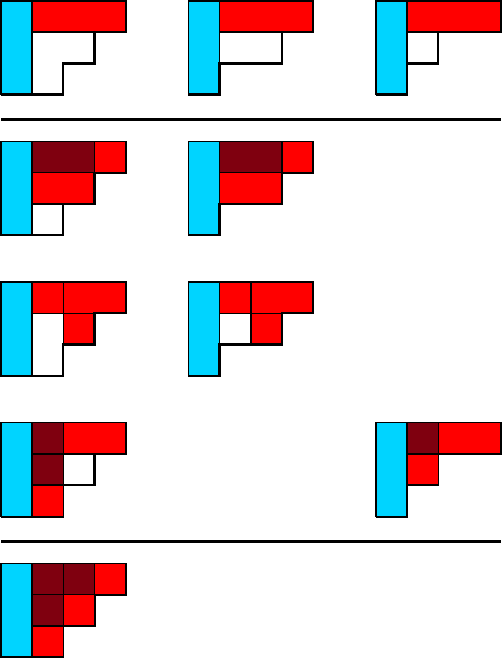}
\end{center}
\caption{Tight gradings in $\mathcal{P}(4,3)$ associated to rearrangements of $\overline{\pt_2} \vdash 3$.  See \Cref{exI3}.  Here, line by line, we see 
(Top Line): $(0,0,3)$, $(0,3,0)$, $(3,0,0)$;
(2nd Line): $(0,2,1)$, $(2,1,0)$;
(3rd Line): $(0,1,2)$, $(1,2,0)$; \\
(4th Line): $(1,0,2)$; $(2,0,1)$; and 
(5th Line): $(1,1,1)$.}
\label{fig:iEq3}
\end{figure}

\section{Final remarks}
\subsection{Combinatorics of tight gradings}
In \Cref{picto_tight}, we described a process for generating a tiling associated to each tight grading.  This tiling is only a partial tiling, in that there can be empty spaces in the region between the Dyck path and the frame.  It could be interesting to characterize the number of such spaces, and in particular study the case of ``tightest gradings'' when no such spaces appear.  Note that this tiling construction is different from that described in \cite{BLM} and \cite{BLM2}, in that the footprints contain a superset of the tiles in earlier constructions.  

As discussed in \Cref{rem: catalan bijection}, the composition of mutation and retractions gives a bijection between the tight gradings in $\CG_{3,2}^{\top}(3i+2,2i+1,i,i)$ and in $\CG_{3,2}^{\vdash}(3i+1,2i+1,i,i)$.  Each set of tight gradings gives a combinatorial interpretation of $\tau^{3,2}(i,i)$.  The two sets of tight gradings correspond to a choice of positive or negative angular momentum broken line in the scattering diagram $\mathfrak{D}_{(3,2)}$, as described in \cite{BLM2}.  It would be interesting to construct a general bijection between the sets of tight gradings in $\CGbc^{\vdash}$ and $\CGbc^{\top}$ that each give a combinatorial interpretation to the same wall-function coefficient.

\subsection{Potential relation with stable trees}
There are several situations where Euler characteristics of quiver moduli spaces can be computed by counting so-called stable trees \cite{Weist}. In fact, this is one of the main techniques in \cite{RW} for obtaining the formulas which are used in Section \ref{section: wall function quiver moduli}. It is thus natural to ask for bijections between sets of stable trees and sets of tight gradings (maybe suitably weighted), but no general pattern emerged yet. One obstacle for such constructions is the fact that the notion of tight gradings is well-adapted to the natural Weyl group symmetry, whereas stable trees are not (the Weyl group symmetry corresponds to applying reflection functors to quiver representations, which does not preserve the class of tree representations naturally).

\subsection{Algebraicity and growth rate}

The conjectural wall-function in \Cref{C13} has the first few terms
\[
    1 + 2x_1^2x_2^5 + 5x_1^4x_2^{10} + 17x_1^6x_2^{15} + 64 x_1^{8}x_2^{20} + 259 x_1^{10}x_2^{25} + \cdots.
\]
Write the generating function as
\[
    S(x) = 1 + 2x + 5x^2 + 17x^3 + \cdots.
\]
Let $A(x) = 1 - xS(x)$. According to \cite[A217596]{OEIS}, it satisfies the algebraic equation
\[
    A(x)^3 = A(x)^2 - xA(x) - x^2.
\]
Then it follows that the coefficients have asymptotic exponential growth rate $27/5$.

Notice that according to the proof of \Cref{prop: central wall}, the wall-function on $\mathbb R_{\leq 0}(b, c)$ is algebraic over $\mathbb {Q}(x_1^{ib}x_2^{jc})$. The coefficients have asymptotic growth rate $\frac{m^m}{(m-1)^{m-1}}$ where $m = (b-1)(c-1)$, assuming $b, c>1$.

It is an interesting question whether every wall-function $\mathfrak{f}_{\mathbb{R}_{\leq 0}(db, ec)}$ with of a cluster scattering diagram is algebraic (first asked by Kontsevich \cite[Section 1.4]{GP}) and whether the coefficients asymptotically always have an exponential growth rate when $bc>4$ and $\mathbb{R}_{\leq 0}(db, ec)$ is in the ``Badlands'', i.e.,
\begin{equation}\label{eq: interval of badlands}
    \frac{e}{d} \in \left(\frac{bc-\sqrt{bc(bc-4)}}{2c}, \frac{bc+\sqrt{bc(bc-4)}}{2c}\right).
\end{equation}

Based on numerical experiments, we propose the following conjecture.

\begin{conj}\label{conj:growth rate}
    For any $r$ in the open interval \eqref{eq: interval of badlands}, the limit
    \[
        f(r)\coloneqq \lim_{\frac{j}{i}\to r,\,i\to \infty} \frac{\ln \tau(i, j)}{\sqrt{bcij - bi^2 - cj^2}}
    \]
    exists and is continuous in $r$.
\end{conj}

By taking $r = e/d$ and the sequence $(i, j) = (kd, ke)$ for $k\geq 1$, \Cref{conj:growth rate} implies an exponential asymptotic growth of the coefficients along the wall $\mathbb{R}_{\leq 0}(db, ec)$. Using the Weyl group symmetry \eqref{eq: weyl symmetry}, one can further show that $f(r)$ is bounded provided continuity.

In the case where $b = c$, \Cref{conj:growth rate} implies a conjecture of Michael Douglas stated in \cite[Conjecture 6.1]{Weist} and \cite[Section 6 (iii)]{GP}, where $\tau(d, e)$ equals the Euler characteristic $\chi(M^s_{d,e}(b))$ defined therein. We remark, however, unlike \cite[Conjecture 6.1]{Weist} that we do not expect $f(r)$ to extend continuously to the endpoints of the interval \eqref{eq: interval of badlands}.

\appendix

\section{\texorpdfstring{Tables of $\lambda(\pt_1, \pt_2)$}{Tables of Lambda(P1, P2)}}

We present tables of $\lambda(\pt_1, \pt_2)$ for $\pt_1\vdash k$, $\pt_2\vdash k$, $2\leq k \leq 6$. Notice that for a given $k$, the total sum of the entries is computed in \Cref{prop: total num tight grading central}. The formulas for the first row (equivalently the first column) of such tables are given by \Cref{prop: formula one length one part}. 
Since $|\pt_1| = |\pt_2|$, the combinatorial argument at the end of Section \ref{sec:countingTG} also applies.  A formula for general $\lambda(\pt_1, \pt_2)$ is currently unknown.

\begin{table}[hb]
\centering
\begin{tabular}{l|cc}
 & (2) & (1,1) \\
\hline
(2)   & 2 & 1  \\
(1,1) & 1 & 0 \\
\end{tabular}
\caption{$k=2$}
\end{table}

\vspace{-0.1em}
\begin{table}[ht]
\centering
\begin{tabular}{l|ccc}
 & (3) & (2,1) & (1,1,1) \\
\hline
(3)     & 3 & 6 & 1 \\
(2,1)   & 6 & 5 & 0 \\
(1,1,1) & 1 & 0 & 0 \\
\end{tabular}
\caption{$k=3$}
\end{table}

\begin{table}[ht]
\centering
\begin{tabular}{l|ccccc}
 & (4) & (3,1) & (2,2) & (2,1,1) & (1,1,1,1) \\
\hline
(4)         & 4 & 12 & 6 & 12 & 1 \\
(3,1)       & 12 & 23 & 10 & 10 & 0 \\
(2,2)       & 6 & 10 & 3 & 4 & 0 \\
(2,1,1)     & 12 & 10 & 4 & 0 & 0 \\
(1,1,1,1)   & 1 & 0 & 0 & 0 & 0 \\
\end{tabular}
\caption{$k=4$}
\end{table}

\begin{table}[ht]
\centering
\begin{tabular}{l|ccccccc}
 & (5) & (4,1) & (3,2) & (3,1,1) & (2,2,1) & (2,1,1,1) & (1,1,1,1,1) \\
\hline
(5)         & 5 & 20 & 20 & 30 & 30 & 20 & 1 \\
(4,1)       & 20 & 59 & 55 & 57 & 51 & 17 & 0 \\
(3,2)       & 20 & 55 & 36 & 48 & 32 & 13 & 0 \\
(3,1,1)     & 30 & 57 & 48 & 25 & 21 & 0 & 0 \\
(2,2,1)     & 30 & 51 & 32 & 21 & 14 & 0 & 0 \\
(2,1,1,1)   & 20 & 17 & 13 & 0 & 0 & 0 & 0 \\
(1,1,1,1,1) & 1 & 0 & 0 & 0 & 0 & 0 & 0 \\
\end{tabular}
\caption{$k=5$}
\end{table}

\begin{table}[ht]
\scriptsize
\centering
\begin{tabular}{l|ccccccccccc}
 & (6) & (5,1) & (4,2) & (4,1,1) & (3,3) & (3,2,1) & (3,1,1,1) & (2,2,2) & (2,2,1,1) & (2,1,1,1,1) & (1,1,1,1,1,1) \\
\hline
(6)       & 6   & 30  & 30  & 60  & 15  & 120 & 60  & 20  & 90  & 30  & 1 \\
(5,1)     & 30  & 119 & 114 & 176 & 57  & 332 & 114 & 52  & 156 & 26  & 0 \\
(4,2)     & 30  & 114 & 86  & 160 & 39  & 228 & 96  & 28  & 102 & 20  & 0 \\
(4,1,1)   & 60  & 176 & 160 & 169 & 78  & 295 & 51  & 44  & 66  & 0   & 0 \\
(3,3)     & 15  & 57  & 39  & 78  & 15  & 102 & 45  & 14  & 45  & 9   & 0 \\
(3,2,1)   & 120 & 332 & 228 & 295 & 102 & 397 & 83  & 48  & 86  & 0   & 0 \\
(3,1,1,1) & 60  & 114 & 96  & 51  & 45  & 83  & 0   & 12  & 0   & 0   & 0 \\
(2,2,2)   & 20  & 52  & 28  & 44  & 14  & 48  & 12  & 4   & 10  & 0   & 0 \\
(2,2,1,1) & 90  & 156 & 102 & 66  & 45  & 86  & 0   & 10  & 0   & 0   & 0 \\
(2,1,1,1,1)& 30 & 26  & 20  & 0   & 9   & 0   & 0   & 0   & 0   & 0   & 0 \\
(1,1,1,1,1,1) & 1 & 0 & 0 & 0 & 0 & 0 & 0 & 0 & 0 & 0 & 0 \\
\end{tabular}
\caption{$k=6$}
\end{table}

\newpage
\addtocontents{toc}{\protect\setcounter{tocdepth}{-1}}

\addtocontents{toc}{\protect\setcounter{tocdepth}{1}}
\enlargethispage{2\baselineskip}

\end{document}